\documentclass[10pt]{article}%
\usepackage{amsmath}
\usepackage{amsfonts}
\usepackage{amssymb}
\usepackage{graphicx}%
\usepackage{mathrsfs}
\usepackage[square, comma, sort&compress, numbers]{natbib}
\usepackage{amssymb, color}
\usepackage[body={15.5cm,21cm}, top=3cm]{geometry}%
\providecommand{\U}[1]{\protect \rule{.1in}{.1in}}
\DeclareMathOperator*{\sgn}{sgn}
\DeclareMathOperator*{\esssup}{ess\,sup}

\newtheorem{theorem}{Theorem}[section]

\newtheorem{lemma}[theorem]{Lemma}

\newtheorem{remark}[theorem]{{Remark}}

\newenvironment{proof}[1][Proof]{\noindent \textbf{#1.} }{\  \rule{0.5em}{0.5em}}
\begin{document}

\title{Quadratic $G$-BSDEs with convex generators
	and unbounded terminal conditions}
\author{ Ying Hu\thanks{Univ. Rennes, CNRS, IRMAR-UMR6625, F-35000, Rennes, France and School of Mathematical Sciences, Fudan University, Shanghai
		200433, China. ying.hu@univ-rennes1.fr.
		Research   supported
		by Lebesgue Center of Mathematics ``Investissements d'avenir"
		program-ANR-11-LABX-0020-01, by CAESARS-ANR-15-CE05-0024 and by MFG-ANR-16-CE40-0015-01.}
	\and Shanjian Tang \thanks{Department of Finance and Control Sciences, School of Mathematical Sciences, Fudan University, Shanghai
		200433, China. sjtang@fudan.edu.cn. Research supported by the National Natural Science Foundation of China (Nos. 11631004 and 12031009).}
	\and Falei Wang\thanks{Zhongtai Securities Institute for Financial  Studies, Shandong University, Jinan 205100, China.
		flwang2011@gmail.com. Research supported by   the National Natural Science Foundation of China (Nos. 12031009 and 11601282) and the Young Scholars Program of Shandong University.}}
\date{}
\maketitle
\begin{abstract}
In this paper, we first study one-dimensional quadratic  backward stochastic differential equations driven by $G$-Brownian motions ($G$-BSDEs) with unbounded terminal values. With the help of a $\theta$-method of Briand and Hu~\cite{BH2008} and nonlinear stochastic analysis techniques, we propose an approximation procedure to  prove existence and uniqueness result when the generator is convex (or concave) and terminal value is of exponential moments of arbitrary order. Finally, we also establish the well-posedness of multi-dimensional $G$-BSDEs with diagonally quadratic generators.
\end{abstract}


\textbf{MSC-classification}: 60H10, 60H30

\section{Introduction}
The present paper is devoted to the study of backward stochastic differential equations (BSDEs) on a $G$-expectation space, which was initiated by Peng \cite{P07a,P08a} motivated by mathematical finance problems with Knightian uncertainty. More precisely, we will investigate the case with  quadratic convex generators  and  unbounded terminal conditions.

The nonlinear BSDE was firstly introduced by Pardoux and Peng~\cite{PP} on  Wiener space $(\Omega, \mathscr{F},P_0)$  with
the natural filtration $(\mathscr{F}_t )_{t\in[0,T ]}$.
The solution of BSDE consists of  a pair  of progressively measurable processes $(Y,Z)$ such that
\begin{align}\label{myq00}
Y_{t}=\xi+\int_{t}^{T}f(s,Y_{s},Z_{s})ds -\int_{t}^{T}Z_{s}dW_{s}, \ \forall t\in[0,T],
\end{align}	
in which $W$ is a standard Brownian motion, and the generator $f$ is a progressively measurable function and the terminal condition $\xi$ is an $\mathscr{F}_T$-measurable
random variable.
 Pardoux and Peng~\cite{PP}  established the existence and uniqueness of solutions to BSDE \eqref{myq00} via a contraction mapping approach when   $f$ is uniformly Lipschitz continuous in both unknowns  and $\xi$ is square integrable.
 Since then, great progress has been made in the field of BSDEs, as it has rich connections with partial differential equations, stochastic control and  mathematical finance (cf. El Karoui et al. \cite{EKP1}). In particular, an extensive study has been given to BSDEs with generators having a quadratic growth in the 2nd unknown $z$ due to their  several financial  motivations, such as utility maximization problems and financial market equilibrium problems  (cf. Hu et al. \cite{HI1}).

For a one-dimensional quadratic BSDE, the monotone convergence method is a  successful strategy to build a solution. Kobylanski \cite{K1}  firstly established the existence and uniqueness theorem through the monotone convergence method and PDE-based approximation technique when the terminal condition is bounded.
Subsequently, Briand and Hu \cite{BH2006} extended the existence result to the case of unbounded terminal conditions. Indeed, they developed a useful a priori estimate on $Y$, which allows to
apply a monotone approximation technique when the terminal condition has exponential moments. However, the uniqueness of unbounded solutions to  quadratic BSDE  is not trivial. In \cite{BH2008}, Briand and Hu formulated a $\theta$-method to obtain the uniqueness when the generator is convex (or concave) with respect to the 2nd unknown  $z$. On the other hand, several efforts have been made towards proposing new methods for the research of quadratic BSDEs.
When the terminal value is bounded, with the help of BMO martingale theory, Tevzadze \cite{Te} obtained the existence and uniqueness result using  a Picard iteration, and Briand and Elie \cite{BH1} gave a distinct approximation procedure  to derive the solvability based on Malliavin calculus.

It is worth mentioning that  the result of Tevzadze  \cite{Te}  still works for multi-dimensional quadratic BSDE with small enough terminal conditions. However, multi-dimensional quadratic BSDE (even with bounded  terminal conditions) may not have  a solution (see Frei and Dos Reis \cite{FR} for such a counterexample). Then, some  structure conditions on the generator are introduced in order to guarantee that the system of quadratic BSDEs with bounded terminal values has a unique solution. For example, Hu and Tang \cite{HT} investigated BSDEs of diagonally quadratic generators (see also \cite{CN,L}).  Recently, by utilizing a $\theta$-method and   iterative technique, Fan et al. \cite{FHT2} established the solvability of system of diagonally quadratic BSDEs with the terminal values of exponential moments of arbitrary order. For more research on this topic, we refer the reader to \cite{BE0,FH,FHT1,HR2016,M,XZ2016} and the references therein.

In this paper, our probabilistic setup is  the  $G$-expectation  space $(\Omega, L^1_G(\Omega),\mathbb{\hat{E}}[\cdot],(\mathbb{\hat{E}}[\cdot])_{t\in[0,T]})$, under which the canonical process $B$ is called $G$-Brownian motion. The $G$-expectation  is a time-consistent sublinear expectation, and we could establish the corresponding stochastic calculus theory with respect to $G$-Brownian motion, such as  $G$-It\^{o}'s formula, $G$-stochastic differential equation and so on. Indeed, the $G$-expectation could be represented by an upper expectation over  a weakly compact subset of mutually singular martingale measures (cf. Denis et al. \cite{DHP11}).

Due to the nonlinear structure, the quadratic process $\langle B\rangle$ is no longer a deterministic process, which results in the main difficulty compared to the linear case. For instance, there is a kind of non-increasing and continuous $G$-martingales $K$. Thus, a typical BSDE driven by $G$-Brownian motion ($G$-BSDE) is given by
\begin{align}\label{Finite time G-BSDE}
Y_{t}=\xi+\int_{t}^{T}f(s,Y_{s},Z_{s})ds+\int_{t}^{T}g(s,Y_{s},Z_{s})d\langle B\rangle_{s} -\int_{t}^{T}Z_{s}dB_{s}-(K_{T}-K_{t}).
\end{align}	
However, the Picard iteration involving the term $Z$  was found difficult to be applied  to $G$-BSDE due to the presence of $K$. Then, Hu et al. \cite{HJPS} turned to a combined PDE and Galerkin approximation
approach to obtain the well-posedness of  Lipschitz $G$-BSDE \eqref{Finite time G-BSDE} when the terminal value has a finite moment of order  $p>1$. On the other hand, the application of monotone convergence theorem is restricted under $G$-expectation framework (see Lemma \ref{downward convergence proposition}).  So, Hu et al. \cite{HL} adapted the approximation approach of \cite{HJPS} to quadratic $G$-BSDE based on $G$-BMO martingale theory, and derived the existence and uniqueness result when the terminal value is bounded.

 A notion quite related  to $G$-BSDE is the second order BSDE (2BSDE)  proposed  by Soner et al. \cite{STZ1, STZ2}.
 By applying  quasi-surely analysis and aggregations approach,
  Soner et al. \cite{STZ1} established the existence and uniqueness of solutions to 2BSDE  with Lipschitz generators. Possama\"{i} and Zhou \cite{PZ} obtained the solvability of  2BSDE  with quadratic generators and bounded terminal conditions.
We would like to mention that the setting of 2BSDE is more general than that of $G$-BSDE, whereas the solution of $G$-BSDE has more regularity, see \cite{HJPS1,HW, Lin, PTZ,S2} and the references therein for more research on this field.

This paper aims  to fill the gap between boundedness and
existence of exponential moments of the data for solution of quadratic $G$-BSDEs.  We have to develop an alternative approximation approach, which is different from  existing monotone approximation and  Picard
approximation. The key point is how to  estimate the difference of two solutions, say $Y$ and $\bar{Y}$.
Contrary to the case of bounded terminal values,  the 2nd unknown  $Z$ may be  unbounded in the BMO space and the conventional linearization technique  fails to work  in our context. Inspired by the arguments of \cite{BH2008} and \cite{FHT2}, we will  develop a $\theta$-method to  our quadratic   $G$-BSDE under the further  assumption of  either convexity  or concavity on the generator, i.e.,
we estimate $Y-\theta\bar{Y}$ for each $\theta\in(0,1)$, which allows to take advantage of the convexity of the generator.

In order to carry out the purpose, we firstly establish a priori estimate on exponential moments of the term $Y$ of $G$-BSDE \eqref{Finite time G-BSDE} as in \cite{BH2006} or \cite{FHT1}. Unlike the quadratic BSDE case, some delicate and technical computations are developed to deal with the new term $K$ through nonlinear stochastic analysis theory, which generalizes the counterpart of \cite{HJPS} (see Lemma \ref{myq10}). Next,  using the decreasing property of the new term $K$ inspired by \cite{CT,PZ}, we give a priori estimate on the term $Z$, which involves exponential moments of the term $Y$.  Then, with the help of a $\theta$-method, we could develop an approximation procedure through a sequence of quadratic $G$-BSDEs
with bounded terminal condition. Indeed, we prove existence and
uniqueness of the global solution to  quadratic $G$-BSDE with the
terminal value of exponential moments of arbitrary order.
Finally, we  consider the solvability of systems of diagonally quadratic $G$-BSDEs with unbounded terminal values, and give some  extension of \cite{Liu}'s result to our quadratic case.

The rest of the paper is organized as follows. In section 2, we present some basic results on $G$-expectation. Section 3 is devoted to solution of quadratic $G$-BSDE with  unbounded terminal conditions. In section 4, we discuss a multi-dimensional case.

\section{The $G$-expectation setup}
In this paper,  we denote by $\left\langle\cdot,\cdot\right\rangle$ and $|\cdot|$ the scalar product and associated norm of a Euclidian space, respectively. Fix a constant $T>0$. Let $\Omega=C_{0}^{d}([0,T])$ be the space of all $\mathbb{R}^{d}%
$-valued continuous functions $\omega$ starting from the origin on $[0,T],$
and $B_t(\omega):=\omega_t$ be the canonical process,
equipped with the uniform norm, i.e.,
$
\|\omega\|:=\sup_{t\in[0,T]}|\omega_{t}|.
$
We set $\Omega_{t}:=\{\omega_{\cdot\wedge t}: \omega\in \Omega\}$ and denote by  $\mathcal{B}(\Omega) \ (resp. \ \mathcal{B}(\Omega_t))$ the Borel $\sigma$-algebra of $\Omega \ (resp. \ \Omega_{t})$ for each $t\in[0,T]$.
We introduce the following space of cylinder functions as
 a counterpart of cylinder sets in the linear case:
\begin{align*}
L_{ip}(\Omega_t):=\{ \varphi(B_{t_{1}},\ldots,B_{t_{k}}):k\in \mathbb{N},\ t_{1}<\cdots<t_{k}\in \lbrack0,t], \ \varphi \in C_{b.Lip}(\mathbb{R}
^{k\times d })\},
\end{align*}
and  $L_{ip}(\Omega):=L_{ip}(\Omega_T),$
where $C_{b.Lip}(\mathbb{R}^{k\times d})$ denotes the space of all
bounded and Lipschitz functions on $\mathbb{R}^{k\times d}$.

Given a monotonic and sublinear function $G:\mathbb{S}(d)\rightarrow \mathbb{R}$, where $\mathbb{S}(d)$ denotes the space of all $d\times d$ symmetric matrices.
 Peng \cite{P07a, P08a} initiated the $G$-expectation $\mathbb{\hat{E}%
}[\cdot]: L_{ip}(\Omega)\rightarrow\mathbb{R}$ satisfying
that
$\mathbb{\hat{E}}[\varphi(B_t)]=u(t,0)$, where $u(t,x)$ is the viscosity solution to the following fully nonlinear PDE with initial condition $u(0,x)=\varphi(x)\in C_{b.Lip}(\mathbb{R}^d)$:
\[
\partial_tu(t,x)-G(\partial^2_{xx}u(t,x))=0, \ \forall (t,x)\in (0,T)\times\mathbb{R}^d.
\]
Moreover, he introduced the conditional $G$-expectation $\mathbb{\hat{E}}_{t}[\cdot]: L_{ip}(\Omega)\rightarrow L_{ip}(\Omega_t)$ for each $t\in[0,T]$.
Let $L_{G}^{p}(\Omega)$ (resp. $L_{G}^{p}(\Omega_t)$) be the completion of $L_{ip}(\Omega)$ (resp. $L_{ip}(\Omega_t)$) under the norm $\mathbb{\hat{E}}[|\cdot|^{p}]^{1/p}$ for each
$p\geq1$. The canonical process $B_t=(B^i_t)_{i=1}^d$ is called a $d$-dimensional $G$-Brownian motion on the $G$-expectation space $(\Omega, L^1_G(\Omega),\mathbb{\hat{E}}[\cdot],(\mathbb{\hat{E}}[\cdot])_{t\in[0,T]})$. For each $1\leq i, j \leq d$,  denote by $\langle B^i, B^j\rangle$   the mutual variation process. Indeed, the $G$-expectation could be regarded as an upper expectation.
\begin{theorem}[\cite{DHP11,HP09}]
	\label{the2.7}  There exists a weakly compact set
	$\mathcal{P}$ of probability
	measures on $(\Omega,\mathcal{B}(\Omega))$, such that
	\[
	\mathbb{\hat{E}}[X]=\sup\limits_{P\in\mathcal{P}}E^{P}[X]\ \ \text{for any $ X\in  {L}_{G}^{1}{(\Omega)}$.}
	\]
 $\mathcal{P}$ is called a set that represents $\mathbb{\hat{E}}$.
\end{theorem}

\begin{remark}
{\upshape Denote by $\mathcal{P}_{max}$ the collection of all  probability measures $P$ on $(\Omega,\mathcal{B}(\Omega))$
	 such that	$E^{P}[X]\leq \mathbb{\hat{E}}[X]$ for any $X\in {L}_{ip}(\Omega)$. Then, $\mathcal{P}_{max}$ is a weakly compact set that represents $\mathbb{\hat{E}}$.
In what follows, we always assume that $\mathcal{P}=\mathcal{P}_{max}$ to give the representation of conditional $G$-expectation. Note that the capacities induced by different weakly compact representation sets coincide with each other, see \cite{HP13}.
 }
\end{remark}

Then the  $G$-expectation $\hat{\mathbb{E}}$ could be generalized in the following way:
\[\hat{\mathbb{E}}[X]=\sup_{P\in\mathcal{P}}E^P[X] \ \text{for each $\mathcal{B}(\Omega)$-measurable function $X$}.\]	
For this $\mathcal{P}$, we define capacity%
\[
c(A):=\sup_{P\in\mathcal{P}}P(A),\ A\in\mathcal{B}(\Omega).
\]
A set $A\in\mathcal{B}(\Omega)$ is polar if $c(A)=0$.  A
property holds $``quasi$-$surely"$ (q.s.) if it holds outside a
polar set. In what follows, we do not distinguish between two random
variables $X$ and $Y$ if $X=Y$ q.s.

Due to the nonlinear structure of $G$-expectation space, we have the following  nonlinear  monotone convergence theorem, which is much more complicated.
\begin{lemma}[\cite{DHP11}]\label{downward convergence proposition} Let $X_n$, $n\geq 1$   be a sequence of  $\mathcal{B}(\Omega)$-measurable random variables.
\begin{description}
	\item[(i)] 	Suppose  $X_n\geq 0$. Then,
	$\mathbb{\hat{E}}[\liminf\limits_{n\rightarrow\infty}X_n]\leq \liminf\limits_{n\rightarrow\infty}\mathbb{\hat{E}}[X_n].$
	\item[(ii)] 		Suppose $X_n \in L^1_G(\Omega)$ are non-increasing. Then,
	$\mathbb{\hat{E}}[X_n]\downarrow\mathbb{\hat{E}}[\lim\limits_{n\rightarrow\infty}X_n].$
	\end{description}
\end{lemma}

Next, we introduce some useful spaces of stochastic processes, which will be used frequently in  stochastic calculus theory with respect to $G$-Brownian motion.  Set
 \begin{align*}
&M_{G}^{0}(0,T):=\bigg\{\eta_{\cdot}=\sum_{i=0}^{n-1}\xi_{t_i}\mathbf{1}_{[t_{i},t_{i+1})}(\cdot): 0=t_0<t_1<\cdots<t_n=T, \ \  \xi_{t_i}\in L_{ip}(\Omega_{t_{i}})\bigg\};\\
&S_{G}^{0}(0,T):=\{h(t,B_{t_{1}\wedge t},\cdot\cdot\cdot,B_{t_{n}\wedge
	t}): 0<t_{1}<t_2<\ldots<t_{n}<T, h\in C_{b.Lip}(\mathbb{R}^{1+n\times d})\};\\
&M_G^p(0,T):=\text{the completion of $M_G^0(0,T)$ under the norm $\mathbb{\hat{E}}\bigg[\int_0^{T}|\cdot|^pdt\bigg]^{\frac{1}{p}}$},\ \forall p\geq 1;\\
&H_G^p(0,T):=\text{the completion of $M_G^0(0,T)$ under the norm $\mathbb{\hat{E}}\bigg[\bigg(\int_0^{T}|\cdot|^2dt\bigg)^{\frac{p}{2}}\bigg]^{\frac{1}{p}}$},\ \forall p\geq 1;\\
&S_G^p(0,T):=\text{the completion of $S_G^0(0,T)$ under the norm $
\mathbb{\hat{E}}\big[\sup_{t\in [0,T]}|\eta_{t}|^{p}\big]^{\frac{1}{p}}$},\ \forall p\geq 1.
\end{align*}
Then, for two processes  $ \eta\in M_{G}^{1}(0,T)$ and $ \xi\in H_{G}^{1}(0,T)$,
the $G$-It\^{o} integrals  $\int\eta_sd\langle B^i,B^j\rangle_s$  and $\int \xi_sdB^i_s$ could be constructed by a standard approximation method.
Denote by $M_G^p(0,T; \mathbb{R}^d)$ the $\mathbb{R}^d$-valued process  such that each component  belongs to $M_G^p(0,T)$. Similarly, we can define
$H_G^p(0,T;\mathbb{R}^d)$, $S^p_G(0,T;\mathbb{R}^d)$ and $L_G^p(\Omega; \mathbb{R}^d)$.  For the sake of convenience, set
 $$\int_{0}^{T}\eta_{s}d\langle B\rangle_s:=\sum_{i,j=1}^{d}\int_0^T\eta^{ij}_{s}d\langle B^i, B^j\rangle_s\  \text{and} \ \int_0^T\xi_sdB_s:=\sum_{i=1}^d\int_0^T\xi^i_sdB^i_s $$ for each $\eta \in M_G^1(0,T;\mathbb{R}^{d\times d})$ and $\xi\in H_G^1(0,T; \mathbb{R}^d)$.

For any $p\ge 1$, we denote by $\mathcal{E}_{G}^p(\mathbb{R}^d)$ the collection
of all stochastic processes $Y$ such that $e^Y\in S_{G}^p(0,T;\mathbb{R}^d)$,   $\mathcal{H}_{G}^p(\mathbb{R}^d)$ the collection
of all stochastic processes $Z\in
H_{G}^p(0,T;\mathbb{R}^{d})$, and $\mathcal{L}_G^p(\mathbb{R}^d)$ the collection
of all stochastic processes $K$ such that $K$ is a non-increasing $G$-martingale with $K_{0}=0$ and $K_{T}\in L_{G}^p(\Omega;\mathbb{R}^d)$. We write $Y\in \mathcal{E}_{G}(\mathbb{R}^d)$ if $Y\in \mathcal{E}_{G}^p(\mathbb{R}^d)$ for any $p\geq 1$. Similarly, we define $\mathcal{H}_{G}(\mathbb{R}^d)$ and $\mathcal{L}_{G}(\mathbb{R}^d)$.

In the rest of this paper, we always assume
that $G$ is non-degenerate to ensure the solvability of $G$-BSDEs, i.e., there exist two constants
$0<\tilde{\sigma}^{-1}\leq\bar{\sigma}<\infty$ such that
\begin{equation*}
\frac{1}{2\tilde{\sigma}^2}\mathrm{tr}[A-B]\leq G(A)-G(B)\leq \frac{1}{2}\bar{\sigma}^{2}\mathrm{tr}[A-B] \ \ \text{ for all }A\geq B.
\end{equation*}
Then it follows from Corollary 3.5.8 of Peng \cite{P10} that
\begin{equation}\label{Non-degenerated condition of G}
\tilde{\sigma}^{-2} I_{d}\leq [d\langle B^i,B^j\rangle_t]_{i,j=1}^d\leq \bar{\sigma}^{2} I_{d}.
\end{equation}
It is easy to verify that $\int \xi_sdB_s$ is a $P$-martingale for each $P\in\mathcal{P}$. Then we have the following BDG inequality.

\begin{lemma}
	\label{myw90156} Assume that $\xi \in {M}^{2}_G(0,T;\mathbb{R}^d)$.  Then, for each $n\geq 1$, there is a constant $A(n)$ depending only on $n$ and $\bar{\sigma}^2$ so that have%
	\begin{align}	\label{myw901}
\mathbb{\hat{E}}\bigg[\sup_{t\in
		\lbrack0,T]}\bigg|\int_{0}^{t}\xi_{s}dB_{s}\bigg|^{n}\bigg]\leq A(n)
	\mathbb{\hat{E}}\bigg[\bigg(\int_{0}^{T}|\xi_{s}|^{2}ds\bigg)^{
\frac{n}{2}}\bigg].
	\end{align}
\end{lemma}
We have Doob's maximal inequality for $G$-martingale as follows.  See \cite{P10, STZ, Song11} for details.
\begin{lemma}\label{myq7}
	Suppose $1<\alpha<\beta$. Then for each $1<p<\bar{p}:=\beta/\alpha$ with $p\leq2$ and for all $\eta\in L_G^{\beta}(\Omega_T)$, there exists a constant $C>0$ depending only on $p$, $\bar{\sigma}$ and $\tilde{\sigma}$ such that
	\[
	\mathbb{\hat{E}}\bigg[\sup_{t\in [0,T]}\mathbb{\hat{E}}_t[|X|^{\alpha}]\bigg] \leq \frac{Cp\bar{p}}{(\bar{p}-p)(p-1)}\bigg(\mathbb{\hat{E}}[|X|^{\beta}]^{\frac{1}{p\bar{p}}} +\mathbb{\hat{E}}[|X|^{\beta}]^{\frac{1}{p}}\bigg).
	\]
\end{lemma}
\begin{remark}\label{myq7}{\upshape
Suppose $e^{X}\in L^2_G(\Omega)$. Then, there exists a constant $\hat{A}(G)$ depending only on  $\bar{\sigma}$ and $\tilde{\sigma}$ such that (taking $\alpha=2,\beta=4$)	\[
\mathbb{\hat{E}}\bigg[\sup_{t\in [0,T]}\mathbb{\hat{E}}_t[e^{X}]\bigg]=\mathbb{\hat{E}}\bigg[\sup_{t\in [0,T]}\mathbb{\hat{E}}_t\bigg[\big(e^{\frac{X}{2}}\big)^2\bigg]\bigg]\leq \hat{A}(G)\mathbb{\hat{E}}[e^{2X}].
\]
}
\end{remark}

\section{One-dimensional quadratic $G$-BSDEs}
In this section, we shall study the well-posedness of solutions to the following scalar-valued quadratic $G$-BSDEs with unbounded terminal values:
\begin{align}\label{myq1}
Y_{t}=\xi+\int_{t}^{T}f(s,\omega_{\cdot\wedge s},Y_{s},Z_{s})ds-\int_{t}^{T}Z_{s}dB_{s}-(K_{T}-K_{t}), \ \forall t\in[0,T].
\end{align}
Throughout the paper, we always fix three positive constants $\lambda$, $\gamma$, $\kappa$, and two nonnegative stochastic processes  $\alpha_t, \beta_t\in M_G^1(0,T)$. Consider the following  assumptions on  the terminal condition and  generator.
\begin{description}
	\item[(H1)] For each  $(t,\omega)\in[0,T]\times\Omega$ and $(y,z),(\bar{y},\bar{z})\in \mathbb{R}\times \mathbb{R}^d$,
	\[
|f(t,\omega,0,0)|\leq \alpha_t(\omega)\ \text{and}\	|f(t,\omega,y,z)-f(t,\omega,\bar{y},\bar{z})|\leq \lambda|y-\bar{y}|+\gamma(1+|z|+|\bar{z}|)|z-\bar{z}|.
	\]
\item[(H2)] There exists a modulus of continuity $w:[0,\infty)\rightarrow[0,\infty)$  such that for each $(y,z)\in \mathbb{R}\times \mathbb{R}^d$ and $(t,\omega), (\bar{t},\bar{\omega})
\in [0,T]\times \Omega$,
	\[
|f(t,\omega,y,z)-f(\bar{t},\bar{\omega},y,z)|\leq w(|t-\bar{t}|+\|\omega-\bar{\omega}\|).
	\]
\item[(H3)]  For each $(t,\omega,y)\in[0,T]\times\Omega\times\mathbb{R}$,
$f(t,\omega,y,\cdot)$ is  either convex or concave.
	
\item[(H4)] Both the terminal value $\xi\in L^1_G(\Omega)$ and  $\int^T_0\alpha_tdt$ have  exponential moments of arbitrary order, i.e.,
\[
\mathbb{\hat{E}}\bigg[\exp\bigg\{p|\xi|+p\int^T_0\alpha_tdt\bigg\}\bigg]<\infty\quad \text{for any $p\geq 1$.}
\]
\end{description}

\begin{remark}
	{\upshape	
Assumptions (H1) and (H2) are also used in~\cite{HL} to study quadratic  $G$-BSDEs when the terminal condition and $f(s,0,0)$ are bounded.
Assumption (H3) allows us to use  a $\theta$-method of  \cite{BH2008} to establish the convergence of our approximating sequences of quadratic $G$-BSDEs. Assumption (H4) relaxes the existing bounded condition on the terminal value and generator.
}
\end{remark}

\begin{remark}
	{\upshape
		Just for  convenience of exposition, the type of quadratic $G$-BSDEs \eqref{myq1} are considered. Our results can be proved to be true for more general quadratic $G$-BSDEs, by slightly finer estimates:
		\begin{align*}
		Y_{t}=\xi+\int_{t}^{T}f(s,\omega_{\cdot\wedge s},Y_{s},Z_{s})ds+\sum\limits_{i,j=1}^d\int_{t}^{T}g_{ij}(s,\omega_{\cdot\wedge s},Y_{s},Z_{s})d\langle B^i, B^i\rangle _{s}-\int_{t}^{T}Z_{s}dB_s-(K_{T}-K_{t}).
		\end{align*}
In this case, we call $g(t,\omega,y,\cdot)=[g_{ij}(t,\omega,y,\cdot)]_{i,j=1}^d$ is   convex (resp. concave ) in $z$, if for any $\theta\in (0,1)$,
\[
g(t,\omega,y,\theta z+(1-\theta) \bar{z})- \theta  g(t,\omega,y,z)-(1-\theta)g(t,\omega,y,\bar{z}) \ \text{is non-positive  (resp. non-negative) definite.}
\]	
}	
\end{remark}

Due to the presence of non-increasing $G$-martingale $K$,
it is difficult to apply the monotone convergence argument
or fixed point method to the study of quadratic $G$-BSDEs.  In what follows, we  will combine nonlinear stochastic analysis technique and $\theta$-method to deal with the question.  Firstly, we need to give some useful lemmas.  We  introduce some general conditions on the generator.
\begin{description}
\item[(H5)]For each  $(t,\omega,y, z)\in[0,T]\times\Omega\times\mathbb{R}\times\mathbb{R}^d$,
\[
|f(t,\omega,y,z)|\leq \beta_t(\omega)+\lambda|y|+\frac{\kappa}{2}|z|^2.
\]
\item[(H6)]For each  $(t,\omega,y, z)\in[0,T]\times\Omega\times\mathbb{R}\times\mathbb{R}^d$,
\[
f(t,\omega,y,z)\mathbf{1}_{\{y>0\}}\leq \beta_t(\omega)+\lambda|y|+\frac{\kappa}{2}|z|^2.
\]
\end{description}
Note that (H1) implies that (H5). Indeed, it follows from Cauchy inequality that
\begin{align}\label{myq771}
|f(t,\omega,y,z)|\leq |f(t,\omega,0,0)|+\lambda|y|+{\gamma}\bigg(|z|^2+\frac{|z|^2+1}{2}\bigg) \leq \alpha_t(\omega)+\frac{\gamma}{2}+\lambda|y|+\frac{3\gamma}{2}|z|^2.
\end{align}

\begin{lemma}\label{myq156} Let $X\in S^1_G(0,T)$ be a stochastic process. Suppose that \[
	\mathbb{\hat{E}}\bigg[\exp\bigg((1+\varepsilon)\sup\limits_{t\in[0,T]}|X_t|\bigg)\bigg]<\infty
	\] for some constant $\varepsilon >0$.
	Then, $\exp(X_t)\in  S^1_G(0,T)$.
\end{lemma}
\begin{proof}
	Denote by $X^{(n)}=(X\wedge n)\vee (-n)$ for each $n\geq 1$. Then, it is easy to check that $\exp(X^{(n)}_t)\in  S^1_G(0,T)$.
	By Taylor's expansion, we have that
	\begin{align*}
	\mathbb{\hat{E}}\bigg[\sup\limits_{t\in[0,T]}\big|\exp(X_t)-\exp(X^{(n)}_t)\big|\bigg]\leq	\mathbb{\hat{E}}\bigg[\sup\limits_{t\in[0,T]}\exp(|X_t|)|X_t|\mathbf{1}_{\{|X_t|\geq n\}}\bigg]\leq \frac{2}{n\varepsilon^2} \mathbb{\hat{E}}\bigg[\sup\limits_{t\in[0,T]}\exp\left((1+\varepsilon)|X_t|\right)\bigg],
	\end{align*}
	which ends the proof by sending $n\rightarrow \infty$.
\end{proof}

Then, we have  the following a priori estimates for quadratic $G$-BSDEs, which is crucial for our subsequent discussions.

\begin{lemma}\label{myq2}
Assume that $(Y,Z,K)\in S^2_G(\mathbb{R})\times\mathcal{H}_G^2(\mathbb{R}^d)\times \mathcal{L}_G^2(\mathbb{R})$  satisfies the following equation
\begin{align*}
Y_{t}=\xi+(\bar{K}_T-\bar{K}_t)+\int_{t}^{T}f(s,Y_{s},Z_{s})ds
-\int_{t}^{T}Z_{s}dB_{s}-(K_{T}-K_{t}),
\end{align*}
where $\bar{K}\in S^2_G(0,T)$ is a non-increasing $G$-martingale. Suppose that there are two constants $p\geq 1$ and $\varepsilon>0$ such that
\begin{align}
\label{myq6}
\mathbb{\hat{E}}\bigg[\exp\bigg\{{(2p+\varepsilon)\kappa\tilde{\sigma}^2}e^{\lambda T}\sup\limits_{t\in[0,T]}|Y_t|+{(2p+\varepsilon)\kappa\tilde{\sigma}^2}\int^T_0\beta_te^{\lambda t}dt\bigg\}\bigg]<\infty.
\end{align}
Then, we have
	\begin{description}
		\item[(i)] Let Assumption (H5) hold. Then, for each $t\in[0,T]$,
	 \begin{align*}&\exp\left\{{p\kappa\tilde{\sigma}^2}e^{\lambda t}|Y_t|\right\}\leq   \mathbb{\hat{E}}_t\bigg[\exp\bigg\{{p\kappa\tilde{\sigma}^2}e^{\lambda T}|\xi|+{p\kappa\tilde{\sigma}^2}\int^T_t\beta_se^{\lambda s}ds\bigg\}	 \bigg].
		\end{align*}
		\item[(ii)] Let Assumption (H6) hold. Then, for each $t\in[0,T]$,
 \begin{align*}
		\exp\left\{{p\kappa\tilde{\sigma}^2}e^{\lambda t}Y_t^+\right\}\leq   \mathbb{\hat{E}}_t\bigg[\exp\bigg\{{p\kappa\tilde{\sigma}^2}e^{\lambda T}\xi^++{p\kappa\tilde{\sigma}^2}\int^T_t\beta_se^{\lambda s}ds\bigg\}
		\bigg].
		\end{align*}	
	\end{description}
\end{lemma}
\begin{proof}
The proof is based on the idea of \cite{FHT1} and nonlinear stochastic analysis technique.

{\bf 1. Proof of Assertion (i). }
From the representation theorem for $G$-expectation (Theorem \ref{the2.7}), we see that $(Y,Z)$ could be regarded as the solution
to the following classical BSDE:
\[
Y_{t}=\xi+(\bar{K}_{T}-\bar{K}_{t})-(K_{T}-K_{t})+\int_{t}^{T}f(s,Y_{s},Z_{s})ds-\int_{t}^{T}Z_{s}dB_{s},\ \ \text{$P$-a.s.}
\]
for each $P\in\mathcal{P}$. Then, applying It\^{o}-Tanaka's formula (see  \cite[p. 234]{RY}) to $|Y_t|$,  we have $P$-a.s.
\begin{align*}
d|Y_t|=-\sgn(Y_t)f(t,Y_{t},Z_{t})dt+\sgn(Y_t)Z_tdB_t+\sgn(Y_t)d(K_t-\bar{K}_t)+dL_t,
\end{align*}
where $\sgn(x)= \mathbf{1}_{\{x> 0\}}- \mathbf{1}_{\{x< 0\}}$ and $L$ is
a continuous adapted and increasing process. It follows that $P$-a.s.
\begin{align*}
d\psi(t,|Y_t|)\geq &-\partial_x\psi(t,|Y_t|)\sgn(Y_t)f(t,Y_{t},Z_{t})dt+
\partial_x\psi(t,|Y_t|)\sgn(Y_t)Z_tdB_t\\
&+\partial_x\psi(t,|Y_t|)\sgn(Y_t)d(K_t-\bar{K}_t)+\frac{1}{2}\partial^2_{xx}\psi(t,|Y_t|)Z_tZ_t^{\top}d\langle B\rangle_t+\partial_t\psi(t,|Y_t|)dt,
\end{align*}
where $\psi(t,x)$ is given by
 \[
\psi(t,x)=\exp\bigg\{{p\kappa\tilde{\sigma}^2}e^{\lambda t}x+{p\kappa\tilde{\sigma}^2}\int^t_0\beta_se^{\lambda s}ds\bigg\}, \ \ (t,x)\in[0,T]\times[0,\infty).
\]
It follows from Assumption (H5) that
\[
-\partial_x\psi(t,|Y_t|)\sgn(Y_t)f(t,Y_{t},Z_{t})\geq -\partial_x\psi(t,|Y_t|)\left(\beta_t+\lambda|Y_t|+\frac{\kappa}{2}|Z_t|^2\right).
\]
Note that $\tilde{\sigma}^{-2} I_{d} \leq d\langle B\rangle_t\leq \bar{\sigma}^2 I_{d}$ according to Inequality \eqref{Non-degenerated condition of G}. In spirit of the following  fact
$$\partial^2_{xx}\psi(t,|Y_t|)\geq 0, $$
we have
\[
\partial^2_{xx}\psi(t,|Y_t|)Z_tZ_t^{\top}d\langle B\rangle_t\geq \tilde{\sigma}^{-2}\partial^2_{xx}\psi(t,|Y_t|)|Z_t|^2dt.
\]
Since $K$ and $\bar{K}$ are  non-increasing and $\partial_x\psi(t,|Y_t|)\geq {p\kappa\tilde{\sigma}^2}$, we derive  $P$-a.s.
\begin{align}\label{myq5}
\begin{split}
&d\psi(t,|Y_t|)\geq \left[\partial_t\psi(t,|Y_t|)-\partial_x\psi(t,|Y_t|)\left( \beta_t+\lambda|Y_t|\right)\right]dt+
\partial_x\psi(t,|Y_t|)\sgn(Y_t)Z_tdB_t\\
&\ \ \ +\frac{1}{2}\left[-\kappa\partial_{x}\psi(t,|Y_t|) +\tilde{\sigma}^{-2}\partial^2_{xx}\psi(t,|Y_t|)\right]|Z_t|^2dt+\partial_x\psi(t,|Y_t|)\left(\mathbf{1}_{\{Y_t>0\}}dK_t+\mathbf{1}_{\{Y_t<0\}}d\bar{K}_t\right)\\
&\  \geq  \frac{1}{2}p(p-1)\kappa^2\tilde{\sigma}^2|Z_t|^2dt+
\partial_x\psi(t,|Y_t|)\sgn(Y_t)Z_tdB_t+\partial_x\psi(t,|Y_t|)\left(\mathbf{1}_{\{Y_t>0\}}dK_t+\mathbf{1}_{\{Y_t<0\}}d\bar{K}_t\right),
\end{split}
\end{align}
where we have used the fact that
\begin{align*}\label{myq3}
\begin{split}
\partial_t\psi(t,x)-\partial_x\psi(t,x)(\lambda x+\beta_t)=0\
\text{and}\ -p\kappa\tilde{\sigma}^2\partial_x\psi(t,x)+
\partial^2_{xx}\psi(t,x)\geq 0.
\end{split}
\end{align*}

In spirit of the condition \eqref{myq6}, we could get that $\partial_x\psi(t,|Y_t|)\in S^2_G(0,T)$ by Lemma \ref{myq156}.  It follows  that for any $P\in\mathcal{P}$, $\int^t_0\partial_x\psi(s,|Y_s|)\sgn(Y_s)Z_sdB_s$ is a $P$-martingale.
 Thus, recalling Equation \eqref{myq5}, we deduce that  $P$-a.s.
{\small\begin{align*}
\psi(t,|Y_t|)+ E^{P}_t\bigg[\int^T_t\partial_x\psi(s,|Y_s|)\mathbf{1}_{\{Y_s>0\}}dK_s+\int^T_t\partial_x\psi(s,|Y_s|)\mathbf{1}_{\{Y_s<0\}}d\bar{K}_s\bigg]\leq E^{P}_t\left[\psi(T,|\xi|)\right],
 \end{align*}}
where $E^{P}_t$ denotes the conditional expectation with respect to $\mathcal{B}(\Omega_t)$.
 Noting  $\psi(T,|\xi|)\in L^1_G(\Omega)$ and
using  Lemma \ref{myq10} and Lemma \ref{myq9} in the Appendix, we conclude that $P$-a.s.
\[
\psi(t,|Y_t|)\leq \esssup\limits_{\bar{P}\in\mathcal{P}(t,P)}E^{\bar{P}}_t[\psi(T,|\xi|)]=\mathbb{\hat{E}}_t\left[\psi(T,|\xi|)\right],
\]
where $\mathcal{P}(t,P)=\{\bar{P}\in\mathcal{P}|\ \bar{P}=P \text{ on $Lip(\Omega_t)$} \}$.
Since $P\in\mathcal{P}$ is arbitrary, we get that \[\psi(t,|Y_t|)\leq \mathbb{\hat{E}}_t\left[\psi(T,|\xi|)\right], \ \text{q.s.}, \] which is the desired result.

{\bf 2. Proof of Assertion (ii).}  Using another type of It\^{o}-Tanaka's formula (see \cite[p. 222]{RY}) to $Y_t^+$, we have  $P$-a.s.
\begin{align*}
dY_t^+=-\mathbf{1}_{\{Y_t>0\}}f(t,Y_{t},Z_{t})dt+\mathbf{1}_{\{Y_t>0\}}Z_tdB_t+\mathbf{1}_{\{Y_t>0\}}d(K_t-\bar{K}_t)+\frac{1}{2}d\widetilde{L}_t,
\end{align*}
where $\widetilde{L}$ is also a continuous adapted and increasing process, and may differ from $L$. Then, in view of \eqref{myq5},
we have that  $P$-a.s.
\begin{align*}
\begin{split}
d\psi(t,Y_t^+) \geq  \frac{1}{2}p(p-1)\kappa^2\tilde{\sigma}^2|Z_t|^2\mathbf{1}_{\{Y_t>0\}}dt+
\partial_x\psi(t,Y_t^+)Z_t\mathbf{1}_{\{Y_t>0\}}dB_t+\partial_x\psi(t,Y_t^+)\mathbf{1}_{\{Y_t>0\}}dK_t.
\end{split}
\end{align*}
Proceeding identically as  to prove Assertion (i),  the proof is complete.
\end{proof}

\begin{remark}
	{\upshape
From It\^{o}-Tanaka's formula  (\cite[p. 222]{RY}), we have $P$-a.s.
	\begin{align*}
	d|Y_t|=-\widetilde{\sgn}(Y_t)f(t,Y_{t},Z_{t})dt+\widetilde{\sgn}(Y_t)Z_tdB_t+\widetilde{\sgn}(Y_t)d(K_t-\bar{K}_t)+d\widetilde{L}_t,
	\end{align*}
	where $\widetilde{\sgn}(x)= \mathbf{1}_{\{x> 0\}}- \mathbf{1}_{\{x\leq 0\}}$. In this case, it is difficult to prove $\int^t_0\widetilde{\sgn}(Y_s)^+dK_s+\widetilde{\sgn}(Y_s)^-d\bar{K}_s$ is a $G$-martingale,  see Lemma \ref{myq10} below. 	
	}
\end{remark}

\begin{lemma}\label{myq10}
Let $(X^i, K^i)$ be in $S^2_G(0,T)\times \mathcal{L}^2_G(\mathbb{R})$, $i=1,2$.
If $X^1>0$, then $P$-a.s.
{\small \[
\esssup\limits_{\bar{P}\in\mathcal{P}(t,P)}E^{\bar{P}}_t\bigg[\int^T_tX^1_s\mathbf{1}_{\{X^2_s>0\}}dK^1_s+\int^T_tX^1_s\mathbf{1}_{\{X^2_s<0\}}d{K}^2_s\bigg]=0
\]	}
for each $P\in\mathcal{P}$.	\end{lemma}
\begin{proof}
	For each $n\geq 1$, we define
	\begin{align*}
	\varphi_n(x) =\begin{cases}
	 1,\ 	&\text{if}\ x\geq \frac{1}{n};\\
	 nx, 	&\text{if}\ x\in (-\frac{1}{n}, \frac{1}{n});\\
	 -1,\ & \text{if}\ x\leq -\frac{1}{n}.\\
	\end{cases}
	\end{align*}
Then one could easily check that  $\varphi_n(X^2_t)X^1_t\in S^2_G(0,T)$ for each $n\geq 1$.
Thus, from~\cite[Lemma 3.4]{HJPS}, we see that $\int^t_0\varphi_n^+(X^2_s)X^1_sdK^1_s+\int^t_0\varphi_n^-(X^2_s)X^1_sdK^2_s$ is a non-increasing $G$-martingale. Note that $\varphi_n^+(x) \uparrow \mathbf{1}_{\{x> 0\}}$ and $\varphi_n^-(x) \uparrow \mathbf{1}_{\{x< 0\}}$ as $n\rightarrow\infty$. From Lemma \ref{myq9} in the Appendix, we derive that for each $P\in\mathcal{P}$
{\small \begin{align*}
&\esssup\limits_{\bar{P}\in\mathcal{P}(t,P)}E^{\bar{P}}_t\bigg[\int^T_tX^1_s\mathbf{1}_{\{X^2_s>0\}}dK^1_s+\int^T_tX^1_s\mathbf{1}_{\{X^2_s<0\}}dK^2_s\bigg]\\
&=\lim\limits_{n\rightarrow\infty}\mathbb{\hat{E}}_t\bigg[\int^T_t\varphi_n^+(X^2_s)X^1_sdK^1_s+\int^T_t\varphi_n^-(X^2_s)X^1_sdK^2_s\bigg]
=0,\ \text{$P$-a.s.}\end{align*}}
which ends the proof.
\end{proof}

\begin{lemma}\label{myq299}
	Assume  that Assumption (H5) holds and $\int^T_0\beta_tdt$ has exponential moments of arbitrary order. Let $(Y, Z, K)\in  \mathcal{E}_G(\mathbb{R})\times\mathcal{H}_G(\mathbb{R}^d)\times \mathcal{L}_G(\mathbb{R})$ be a solution to  $G$-BSDE \eqref{myq1}.
	Then, for any $n\geq 1$, there exists a constant $\tilde{A}(n)$ depending  on $\lambda,\bar{\sigma},\tilde{\sigma},\kappa, T$ and $n$, such that	
	\[\mathbb{\hat{E}}\bigg[\bigg(\int^T_0|Z_t|^2dt\bigg)^n+|K_T|^n \bigg]\leq \tilde{A}(n)\mathbb{\hat{E}}\bigg[\exp\bigg\{(4\kappa\tilde{\sigma}^2+2\lambda )n \sup\limits_{t\in[0,T]}|Y_t|+2n\int^T_0\beta_tdt\bigg\}\bigg].
	\]
\end{lemma}

\begin{proof}
	Recalling Inequality \eqref{Non-degenerated condition of G} and applying  $G$-It\^{o}'s formula to $e^{-2\kappa \tilde{\sigma}^2Y_t}$ yields that
	\begin{align*}
	&2\kappa^2\tilde{\sigma}^2\int^T_0 e^{-2\kappa \tilde{\sigma}^2Y_t}|Z_t|^2dt\leq 2\kappa^2\tilde{\sigma}^4\int^T_0 e^{-2\kappa \tilde{\sigma}^2Y_t}Z_tZ^{\top}_td\langle B\rangle_t
	\\ \ \ \  &\leq  e^{-2\kappa\tilde{\sigma}^2 \xi} -2\kappa\tilde{\sigma}^2\int^T_0e^{-2\kappa\tilde{\sigma}^2 Y_t}f(t,Y_t,Z_t)dt+2\kappa\tilde{\sigma}^2\int^T_0e^{-2\kappa\tilde{\sigma}^2 Y_t}(Z_tdB_t+dK_t)\\
	\ \ \ & \leq  e^{-2\kappa\tilde{\sigma}^2 \xi}+2\kappa\tilde{\sigma}^2\int^T_0e^{-2\kappa\tilde{\sigma}^2 Y_t}\bigg(\beta_t+\lambda|Y_t|+\frac{\kappa}{2}|Z_t|^2\bigg)dt
	+2\kappa\tilde{\sigma}^2\int^T_0e^{-2\kappa\tilde{\sigma}^2 Y_t}Z_tdB_t,
	\end{align*}
	where we have used the fact that $K$ is a non-increasing process in the last inequality. Thus, in view of the fact that $e^{|x|}\geq |x|$, we could derive that
	\begin{align*}
	\kappa\int^T_0 e^{-2\kappa\tilde{\sigma}^2 Y_t}|Z_t|^2dt
	&\leq  \kappa^{-1}\tilde{\sigma}^{-2}e^{-2\kappa\tilde{\sigma}^2 \xi}+2\int^T_0e^{-2\kappa\tilde{\sigma}^2 Y_t}(\beta_t+\lambda|Y_t|)dt
	+2\int^T_0e^{-2\kappa\tilde{\sigma}^2 Y_t}Z_tdB_t\\
	&\leq 2X+2\int^T_0e^{-2\kappa\tilde{\sigma}^2 Y_t}Z_tdB_t,
	\end{align*}
where $X:=(1+\kappa^{-1}\tilde{\sigma}^{-2}+T)\exp\big\{(2\kappa\tilde{\sigma}^2+\lambda)\sup_{t\in[0,T]}|Y_t|+\int^T_0\beta_tdt\big\}$.
	It follows that
	\begin{align}\label{myq9669}
	\begin{split}
	\kappa^{n}\mathbb{\hat{E}}\bigg[\bigg(\int^T_0 e^{-2\kappa\tilde{\sigma}^2  Y_t}|Z_t|^2dt\bigg)^n\bigg]
	\leq 2^{2n-1}\mathbb{\hat{E}}[X^n]+2^{2n-1}\mathbb{\hat{E}}\bigg[\bigg|\int^T_0e^{-2\kappa\tilde{\sigma}^2 Y_t}Z_tdB_t\bigg|^n\bigg].
	\end{split}
	\end{align}
Applying BDG inequality \eqref{myw901}, for each $n\geq 1$, we can find a constant ${A}(n)$ depending only on $n$ and $\bar{\sigma}^2$ so that
	\begin{align*}
	\begin{split}
	&\mathbb{\hat{E}}\bigg[\bigg|\int^T_0e^{-2\kappa\tilde{\sigma}^2 Y_t}Z_tdB_t\bigg|^n\bigg]
	\leq  {A}(n) \mathbb{\hat{E}}\bigg[\exp\big\{n\kappa\tilde{\sigma}^2 \sup\limits_{t\in[0,T]}|Y_t|\big\}\bigg(\int^T_0e^{-2\kappa\tilde{\sigma}^2 Y_t}|Z_t|^2dt\bigg)^{\frac{n}{2}}\bigg],
	\end{split}
	\end{align*}
which together with the inequality $ab\leq \frac{1}{2\varepsilon} a^2+\frac{1}{2}\varepsilon b^2$ indicates that
	\begin{align} \label{myq991}
	\begin{split}
	&2^{2n-1}\mathbb{\hat{E}}\bigg[\bigg|\int^T_0e^{-2\kappa\tilde{\sigma}^2 Y_t}Z_tdB_t\bigg|^n\bigg]
	\\
	&\leq \frac{2^{4n-3}{A}^2(n)}{\kappa^{n}}\mathbb{\hat{E}}\bigg[\exp\big\{2n\kappa\tilde{\sigma}^2 \sup\limits_{t\in[0,T]}|Y_t|\big\}\bigg]+
	\frac{1}{2}\kappa^{n}\mathbb{\hat{E}}\bigg[\bigg(\int^T_0 e^{-2\kappa\tilde{\sigma}^2  Y_t}|Z_t|^2dt\bigg)^n\bigg].
	\end{split}
	\end{align}
Putting \eqref{myq9669} and \eqref{myq991} together, we could derive that
\begin{align*}
\begin{split}
&\mathbb{\hat{E}}\bigg[\bigg(\int^T_0 e^{-2\kappa\tilde{\sigma}^2  Y_t}|Z_t|^2dt\bigg)^n\bigg]
\leq \tilde{A}_1(n)\mathbb{\hat{E}}\bigg[\exp\bigg\{(2\kappa\tilde{\sigma}^2+\lambda )n \sup\limits_{t\in[0,T]}|Y_t|+n\int^T_0\beta_tdt\bigg\}\bigg],
\end{split}
\end{align*}	
where $\tilde{A}_1(n)$  is given by \[
\tilde{A}_1(n)=\frac{2^{2n}(1+\kappa^{-1}\tilde{\sigma}^{-2}+T)^n\kappa^{n}+{2^{4n-2}{A}^2(n)}}{\kappa^{2n}}, \ \forall n\geq 1.
\]
It follows from H\"{o}lder's inequality that for any $n\geq 1$
\begin{align}\label{myq6789}
\begin{split}	
\mathbb{\hat{E}}\bigg[\bigg(\int^T_0 |Z_t|^2dt\bigg)^n\bigg]
&\leq \mathbb{\hat{E}}\bigg[\exp\bigg\{2n\kappa\tilde{\sigma}^2 \sup\limits_{t\in[0,T]}|Y_t|\bigg\}\bigg(\int^T_0 e^{-2\kappa\tilde{\sigma}^2  Y_t}|Z_t|^2dt\bigg)^n\bigg]\\
&\leq  \mathbb{\hat{E}}\bigg[\exp\bigg\{4n\kappa\tilde{\sigma}^2 \sup\limits_{t\in[0,T]}|Y_t|\bigg\}\bigg]^{\frac{1}{2}}\mathbb{\hat{E}}\bigg[\bigg(\int^T_0 e^{-2\kappa\tilde{\sigma}^2  Y_t}|Z_t|^2dt\bigg)^{2n}\bigg]^{\frac{1}{2}}\\
&\leq \sqrt{\tilde{A}_1(2n)}\mathbb{\hat{E}}\bigg[\exp\bigg\{(4\kappa\tilde{\sigma}^2+2\lambda )n \sup\limits_{t\in[0,T]}|Y_t|+2n\int^T_0\beta_tdt\bigg\}\bigg].	
\end{split}
\end{align}	
	
On the other hand, from \eqref{myq1} and using Assumption (H5), we have that
	\begin{align*}
	-K_T\leq
(2+\lambda T)\sup\limits_{t\in[0,T]}|Y_t|+\int^T_0\beta_tdt+\frac{\kappa}{2}\int^T_0|Z_t|^{2}dt+\int^T_0 Z_tdB_t,
	\end{align*}
Using BDG inequality \eqref{myw901} again, we get that for each $n\geq 1$
	\begin{align*}
	\mathbb{\hat{E}}[|K_T|^n]
	&\leq 4^{n-1}\mathbb{\hat{E}}\bigg[(2+\lambda T)^n\sup\limits_{t\in[0,T]}|Y_t|^n+\bigg(\int^T_0\beta_tdt\bigg)^n\bigg]+2^{n-2}{\kappa}^n\mathbb{\hat{E}}\bigg[\bigg(\int^T_0|Z_t|^2dt\bigg)^n\bigg]\\
	&\ \ \ \ +4^{n-1}{A}(n)\mathbb{\hat{E}}\bigg[\bigg(\int^T_0|Z_t|^2dt\bigg)^{\frac{n}{2}}\bigg].
	\end{align*}
Consequently, by \eqref{myq6789} and in  view of
	the fact that  $e^{x}\geq 1+\frac{|x|^n}{n!}$ for any $x\geq 0$,
	we could find a constant $\tilde{A}(n)$  depending only on $\lambda,\tilde{\sigma},\bar{\sigma},\kappa, T$, and $n$, such that	
	\[\mathbb{\hat{E}}\bigg[\bigg(\int^T_0|Z_t|^2dt\bigg)^n+|K_T|^n \bigg]\leq \tilde{A}(n)\mathbb{\hat{E}}\bigg[\exp\bigg\{(4\kappa\tilde{\sigma}^2+2\lambda )n \sup\limits_{t\in[0,T]}|Y_t|+2n\int^T_0\beta_tdt\bigg\}\bigg],
	\]
	which ends the proof.
\end{proof}

Next, using a $\theta$-method formulated by \cite{BH2008},  we could get  the comparison theorem for quadratic $G$ BSDEs with unbounded terminal values.
\begin{lemma}
	\label{compare theorem}
	Let $(Y^l,Z^l,K^l)$ be a $ \mathcal{E}_G(\mathbb{R})\times\mathcal{H}_G(\mathbb{R}^d)\times \mathcal{L}_G(\mathbb{R})$-solution to  $G$-BSDEs \eqref{myq1} with data $(\xi^l,f^l)$, $l=1,2$. Suppose $(\xi^1,f^1)$ (resp. ($\xi^2,f^2$)) verifies Assumptions (H1)-(H4).
	If $\xi^1\leq \xi^2$ and $ f^1(s,y,z)\leq f^2(s,y,z) $, then $Y^1_t\leq Y^2_t$ for any $t\in[0,T]$.
\end{lemma}
\begin{proof}
 Without loss of generality, assume that $(\xi^1,f^1)$ verifies Assumptions  (H1)-(H4), and the other case could be proved in a similar way.

First, we consider the case when $f^1$ is convex in $z$. 	
	For each $\theta\in(0,1)$, we set
	\begin{align*}
	(\delta_{\theta}Y,\delta_{\theta}Z):=\left(\frac{Y^1-\theta {Y}^2}{1-\theta},\frac{Z^1-\theta {Z}^2}{1-\theta}\right),
	\end{align*}
	and then the triple  $(\delta_{\theta}Y,\delta_{\theta}Z,\frac{1}{1-\theta}K^1)$
	satisfies  the following $G$-BSDE on the interval $[0,T]$:
	\begin{align*}
	\begin{split}
	\delta_{\theta}Y_t=&\delta_{\theta}Y_T+\frac{\theta}{1-\theta}(K^2_T-K^2_t)+\int^T_t\delta_{\theta}f(s,\delta_{\theta}Y_s,\delta_{\theta}Z_s)ds  -\int^T_t\delta_{\theta}Z_sdB_s-\frac{1}{1-\theta}(K_T^1-K^1_t),
	\end{split}
	\end{align*}
	with $\delta_{\theta}f(t,y,z)=\frac{1}{1-\theta}\left(
	f^1(t,(1-\theta)y+{\theta}Y^2_t, (1-\theta)z+{\theta}Z^2_t)-\theta f^2(t,Y^2_t, Z^2_t)\right)$.
It is obvious that $\delta_{\theta}Y_T=\xi^1+\frac{\theta\xi^1-\theta\xi^2}{1-\theta}\leq \xi^1$.
With the help of \eqref{myq771}, we derive  that \[
	|f^1(t,\omega,y,z)|\leq \alpha_t(\omega)+\frac{\gamma}{2}+\lambda|y|+\frac{3\gamma}{2}|z|^2,
	\]
which together with convexity indicates that
	\begin{align*}
	& \delta_{\theta}f(t,y,z)\leq \frac{1}{1-\theta}\left(
	f^1(t,(1-\theta)y+{\theta}Y^2_t, (1-\theta)z+{\theta}Z^2_t)-\theta f^1(t,Y^2_t, Z^2_t)\right) \\
		&\ \ \  \ \ \ \ \ \ \ \ \ \  \ \leq  \lambda|y|+\lambda|{Y}^2_t|+\frac{1}{1-\theta}\left(
		f^1(t,Y^2_t, (1-\theta)z+{\theta}Z^2_t)-\theta f^1(t,Y^2_t, Z^2_t)\right)\\
	&\ \ \  \ \ \ \ \ \ \ \ \ \  \ \leq \lambda|y|+\lambda|{Y}^2_t|+
	f^1(t,Y^2_t, z)\leq \alpha_t+\frac{\gamma}{2}+2\lambda|{Y}^2_t|+\lambda|y|+\frac{3\gamma}{2}|z|^2.
	\end{align*}
	Using Assertion (ii) of Lemma \ref{myq2} (taking $p=1$, $\beta_t=\alpha_t+\frac{\gamma}{2}+2\lambda|{Y}^2_t|$ and $\kappa=3\gamma$), we deduce that
	\begin{align*}	
	\begin{split}
	&\exp\left\{{3\gamma\tilde{\sigma}^2}e^{\lambda t}\left(\delta_{\theta}Y_t\right)^+\right\}\leq   \mathbb{\hat{E}}_t\bigg[\exp\bigg\{{3\gamma\tilde{\sigma}^2}e^{\lambda T}\bigg(|\xi^1|+\frac{\gamma T}{2}+\int^T_t\left(\alpha_s+2\lambda|Y^2_s|\right)ds\bigg)\bigg\}	 \bigg],
	\end{split}
	\end{align*}
	which
	implies that for every $\theta\in (0,1)$ and $t\in[0,T]$,
	\begin{align*}
	&3\gamma\tilde{\sigma}^2(Y^1_t-Y^2_t)^+ \leq	3\gamma\tilde{\sigma}^2(Y^1_t-\theta Y^2_t)^++3(1-\theta)\gamma\tilde{\sigma}^2( Y^2_t)^-\\
	&\leq (1-\theta)\bigg(\mathbb{\hat{E}}_t\bigg[\exp\bigg\{{3\gamma\tilde{\sigma}^2}e^{\lambda T}\bigg(|\xi^1|+\frac{\gamma T}{2}+\int^T_t\left(\alpha_s+2\lambda|Y^2_s|\right)ds\bigg)\bigg\}	 \bigg]+3\gamma\tilde{\sigma}^2(Y^2_t)^-\bigg).
	\end{align*}
	Sending $\theta\rightarrow 1$ yields that $Y^1_t\leq Y^2_t$ for any $t\in[0,T]$.
	
Next, for the case that
 $f^1$ is concave in $z$,  we need to use
	$\theta \ell^{1}- \ell^{2}$ instead of
	$\ell^{1}-\theta \ell^{2}$  in the definition of $\delta_{\theta}\ell$ for $\ell=Y,Z$. In this case, the triple  $(\delta_{\theta}Y,\delta_{\theta}Z,\frac{\theta}{1-\theta}K^{1})$ solves the following $G$-BSDE on $[0,T]$:
	\begin{align*}
	\begin{split}
	\delta_{\theta}Y_t=&\delta_{\theta}Y_T+\frac{1}{1-\theta}(K_T^{2}-K_t^{2})+\int^T_t\delta_{\theta}f(s,\delta_{\theta}Y_s,\delta_{\theta}Z_s)ds  -\int^T_t\delta_{\theta}Z_sdB_s-\frac{\theta}{1-\theta}(K_T^{1}-K_t^{1})
	\end{split}
	\end{align*}
	with
	\begin{align*}
	\delta_{\theta}f(t,y,z)=\frac{1}{1-\theta}\left(
	\theta f^1(t,Y^{1}_t, Z^{1}_t)- f^2(t,-(1-\theta)y+\theta Y^{1}_t, -(1-\theta)z+\theta Z^{1}_t)\right).
	\end{align*}
	By Assumptions (H1), it is easy to check that $	\delta_{\theta}Y_T\leq \frac{\theta\xi^{1}-\xi^{1}}{1-\theta}=-\xi^1$ and
	\begin{align*}
	& \delta_{\theta}f(t,y,z)\leq  \lambda|y|+\lambda|{Y}^1_t|+\frac{1}{1-\theta}\left(
\theta	f^1(t,Y^1_t, Z^1_t)- f^1(t,Y^1_t, -(1-\theta)z+\theta Z^1_t)\right)\\
	&\ \ \  \ \ \ \ \ \ \ \ \ \  \ \leq \lambda|y|+\lambda|{Y}^1_t|
	-f^1(t,Y^1_t, -z)\leq \alpha_t+\frac{\gamma}{2}+2\lambda|{Y}^1_t|+\lambda|y|+\frac{3\gamma}{2}|z|^2.
	\end{align*}
Thus, in view of Assertion (ii) of Lemma \ref{myq2}, we have that
for every $\theta\in (0,1)$ and $t\in[0,T]$,
\begin{align*}
&3\gamma\tilde{\sigma}^2(Y^1_t-Y^2_t)^+ \leq	3\gamma\tilde{\sigma}^2(\theta Y^1_t- Y^2_t)^++3(1-\theta)\gamma\tilde{\sigma}^2( Y^1_t)^+\\
&\leq (1-\theta)\bigg(\mathbb{\hat{E}}_t\bigg[\exp\bigg\{{3\gamma\tilde{\sigma}^2}e^{\lambda T}\bigg(|\xi^1|+\frac{\gamma T}{2}+\int^T_t\left(\alpha_s+2\lambda|Y^1_s|\right)ds\bigg)\bigg\}	 \bigg]+3\gamma\tilde{\sigma}^2(Y^1_t)^+\bigg),
\end{align*}
which completes the proof by sending $\theta\rightarrow 1$.		
\end{proof}

Now we are ready to state the main result of this section, which involves the existence and uniqueness of  unbounded solutions to  quadratic $G$-BSDE \eqref{myq1}.
\begin{theorem}\label{myq66}
Assume that (H1)-(H4) are satisfied. Then, $G$-BSDE  \eqref{myq1} admits a unique solution $(Y, Z, K) \in   \mathcal{E}_G(\mathbb{R})\times\mathcal{H}_G(\mathbb{R}^d)\times \mathcal{L}_G(\mathbb{R})$.
\end{theorem}
\begin{proof}
The uniqueness is immediate from Lemma \ref{compare theorem}. Indeed,
let $(Y^l,Z^l,K^l)$ be a  $ \mathcal{E}_G(\mathbb{R})\times\mathcal{H}_G(\mathbb{R}^d)\times \mathcal{L}_G(\mathbb{R})$-solution
to $G$-BSDE \eqref{myq1}, $l=1,2$. It follows from Lemma \ref{compare theorem} that $Y^1={Y}^2$. Then, applying $G$-It\^o's formula to $\left|Y^1-{Y}^2\right|^2$ yields that $Z^1=Z^2$ and then $K^1=K^2$. Thus, we only need to prove the existence.
We will  construct a solution through a sequence of  quadratic $G$-BSDEs with bounded terminal value and generator. The proof will be divided into the following three steps.
	
{\bf 1. The uniform estimates.} 	
Denote by $f_0(t)=f(t,0,0)$ for narrative  convenience. Then, for each positive integer $m\geq 1$, set
 $$
 \ell^{(m)}=(\ell\wedge m)\vee (-m)\ \text{for $\ell=\xi, f_0(t)$, and}\
f^{(m)}(t, y, z)=f(t,y,z)-f_0(t)+f_0^{(m)}(t).$$
One can easily check that the terminal value $\xi^{m}$ and generator $f^{(m)}$ satisfies \cite[Assumption 2.14]{HL}. Thus,
it follows from \cite[Theorem 5.3]{HL} that,
the $G$-BSDE \eqref{myq1} with data $(\xi^{(m)}, f^{(m)})$ admits a unique solution $(Y^{(m)}, Z^{(m)}, K^{(m)}) \in   \mathcal{E}_G(\mathbb{R})\times\mathcal{H}_G(\mathbb{R}^d)\times \mathcal{L}_G(\mathbb{R})$. From Assumption (H1),  we have that 	\begin{align}
\label{myq40} |f^{(m)}(t,\omega,y,z)|\leq \alpha_t(\omega)+\frac{\gamma}{2}+\lambda|y|+\frac{3\gamma}{2}|z|^2.
\end{align}
In spirit of Lemma \ref{myq2} and Remark \ref{myq7} (taking  $\beta_t=\alpha_t+\frac{\gamma}{2}$ and $\kappa=3\gamma$), we get that for any $p\geq 1$
	\begin{align}
	\label{myq11}
\begin{split}
\sup\limits_{m\geq 1}\mathbb{\hat{E}}\bigg[\exp\bigg\{{3p\gamma\tilde{\sigma}^2}\sup\limits_{0\leq t\leq T}|Y^{(m)}_t|\bigg\}
\bigg] &\leq\mathbb{\hat{E}}\bigg[\sup\limits_{0\leq t\leq T}\mathbb{\hat{E}}_t\bigg[\exp\bigg\{3p\gamma\tilde{\sigma}^2e^{\lambda T}\bigg(|\xi|+\frac{\gamma T}{2}+\int^T_0\alpha_tdt\bigg)\bigg\}	\bigg]\bigg]\\
&\leq \hat{A}(G)  \mathbb{\hat{E}}\bigg[\exp\bigg\{6p\gamma\tilde{\sigma}^2e^{\lambda T}\bigg(|\xi|+\frac{\gamma T}{2}+\int^T_0\alpha_tdt\bigg)\bigg\}	\bigg].
\end{split}
\end{align}	
Recalling Lemma \ref{myq299} and H\"{o}lder's inequality, we obtain for any $n\geq 1$,
\begin{align*}
&\mathbb{\hat{E}}\left[\bigg(\int^T_0|Z^{(m)}_t|^2dt\bigg)^n+|K^{(m)}_T|^n \right]\leq \tilde{A}(n)\mathbb{\hat{E}}\bigg[\exp\bigg\{(12\gamma\tilde{\sigma}^2+2\lambda )n \sup\limits_{t\in[0,T]}|Y_t^{(m)}|+n\gamma T +2n\int^T_0\alpha_tdt\bigg\}\bigg]\\ & \ \ \ \ \ \ \ \ \ \ \ \ \ \ \ \ \ \ \ \ \ \ \
\leq\tilde{A}(n)e^{n\gamma T}\mathbb{\hat{E}}\bigg[\exp\bigg\{(24\gamma\tilde{\sigma}^2+4\lambda )n\sup\limits_{t\in[0,T]}|Y_t^{(m)}|\bigg\}\bigg]^{\frac{1}{2}}\mathbb{\hat{E}}\bigg[\exp\bigg\{4n\int^T_0\alpha_tdt\bigg\}\bigg]^{\frac{1}{2}},
\end{align*}
where the constant $\tilde{A}(n)$ is  independent of $m$. Therefore, in view of \eqref{myq11}, we conclude that
\begin{align} \label{myq30}
&\sup\limits_{m\geq 1}\mathbb{\hat{E}}\left[\bigg(\int^T_0|Z^{(m)}_t|^2dt\bigg)^n+|K^{(m)}_T|^n \right]<\infty, \ \forall n\geq 1.
\end{align}

{\bf 2. $\theta$-method.}  We first consider the case when the generator $f$ is convex in $z$.
For each fixed $m,q\geq 1$ and $\theta\in(0,1)$, we define
\begin{align*}
\delta_{\theta}Y^{(m,q)}:=\frac{Y^{(m+q)}-\theta Y^{(m)}}{1-\theta},\
\delta_{\theta}Z^{(m,q)}:=\frac{Z^{(m+q)}-\theta Z^{(m)}}{1-\theta}.
\end{align*}
Then, the triple $(\delta_{\theta}Y^{(m,q)},\delta_{\theta}Z^{(m,q)},\frac{1}{1-\theta}K^{(m)})$
solves the following $G$-BSDE:
\begin{align}\label{myq12}
\begin{split}
\delta_{\theta}Y^{(m,q)}_t=&\delta_{\theta}\xi^{(m,q)}+\frac{\theta}{1-\theta}(K_T^{(m)}-K_t^{(m)})-\int^T_t\delta_{\theta}Z^{(m,q)}_sdB_s-\frac{1}{1-\theta}(K_T^{(m+q)}-K_t^{(m+q)}) \\
&\ \ +\int^T_t\left(\delta_{\theta}f^{(m,q)}(s,\delta_{\theta}Y^{(m,q)}_s,\delta_{\theta}Z^{(m,q)}_s)+\delta_{\theta}f_0^{(m,q)}(s)\right)ds,
\end{split}
\end{align}
where the terminal value and generator are given by
\begin{align*}
&\delta_{\theta}\xi^{(m,q)}=\frac{\xi^{(m+q)}-\theta\xi^{(m)}}{1-\theta},\ \ \ \delta_{\theta}f_0^{(m,q)}(t)=
\frac{1}{1-\theta}\left(f^{(m+q)}_0(t)-\theta f^{(m)}_0(t)\right)-f_0(t),\\
&
\delta_{\theta}f^{(m,q)}(t,y,z)=\frac{1}{1-\theta}\left(
f(t,(1-\theta)y+{\theta}Y^{(m)}_t, (1-\theta)z+{\theta}Z^{(m)}_t)-\theta f(t,Y^{(m)}_t, Z^{(m)}_t)\right).
\end{align*}
 A direct computation yields that
 \begin{align*}
&\delta_{\theta}\xi^{(m,q)}=\xi^{(m)}+\frac{1}{1-\theta}(\xi^{(m+q)}-\xi^{(m)})\leq |\xi|+\frac{1}{1-\theta}(|\xi|-m)^+,\\
 &\delta_{\theta}f_0^{(m,q)}(t)\leq f_0^{(m)}(t)-f_0(t)+\frac{1}{1-\theta}(|f_0(t)|-m)^+\leq \frac{2}{1-\theta}(|f_0(t)|-m)^+.
 \end{align*}
Using Assumptions (H1) and (H3), we conclude that
 \begin{align*}
\delta_{\theta}f^{(m,q)}(t,y,z)
 \leq  \lambda|y|+\lambda|Y^{(m)}_t|+f(t,Y^{(m)}_t, z)
 \leq  \alpha_t+\frac{\gamma}{2}+2\lambda|Y^{(m)}_t|+\lambda|y|+\frac{3\gamma}{2}|z|^2.
\end{align*}
Applying Assertion (ii) of Lemma \ref{myq2} to Equation \eqref{myq12}, we derive that  for any $p\geq 1$
  \begin{align}	\label{myq21}
 \begin{split}
 &\exp\left\{{3p\gamma\tilde{\sigma}^2}e^{\lambda t}\left(\delta_{\theta}Y^{(m,q)}_t\right)^+\right\}\\
 &\leq   \mathbb{\hat{E}}_t\bigg[\exp\bigg\{{3p\gamma\tilde{\sigma}^2}e^{\lambda T}\bigg(\rho(\theta,m)+|\xi|+\frac{\gamma T}{2}+\int^T_t\left(\alpha_s+2\lambda|Y^{(m)}_s|\right)ds\bigg)\bigg\}	\bigg],
 \end{split}
	\end{align}
where $\rho(\theta,m)$ is given by
\[
\rho(\theta,m):=\frac{1}{1-\theta}(|\xi|-m)^++\frac{2}{1-\theta}\int^T_0(|f_0(t)|-m)^+dt.
\]

On the other hand, we define
\[
\delta_{\theta}\widetilde{Y}^{(m,q)}:=\frac{Y^{(m)}-\theta Y^{(m+q)} }{1-\theta},\
\delta_{\theta}\widetilde{Z}^{(m,q)}:=\frac{Z^{(m)}-\theta Z^{(m+q)}}{1-\theta}.
\]
Then, by a similar analysis, we conclude that
 \begin{align}
 	\label{myq20}
\begin{split}
&\exp\left\{{3p\gamma\tilde{\sigma}^2}e^{\lambda t}\left(\delta_{\theta}\widetilde{Y}^{(m,q)}_t\right)^+\right\}\\
&\leq   \mathbb{\hat{E}}_t\bigg[\exp\bigg\{3p{\gamma\tilde{\sigma}^2}e^{\lambda T}\bigg(\rho(\theta,m)+|\xi|+\frac{\gamma T}{2}+\int^T_t\left(\alpha_s+2\lambda|Y^{(m+p)}_s|\right)ds\bigg)\bigg\}	\bigg].
\end{split}
\end{align}

Note that
\begin{align*}
\left(\delta_{\theta}{Y}^{(m,q)}\right)^-\leq \frac{\theta\left( Y^{(m)}-\theta Y^{(m+q)}\right)^++(1-\theta^2)|Y^{(m+q)}|}{1-\theta}
\leq \left(\delta_{\theta}\widetilde{Y}^{(m,q)}\right)^++2|Y^{(m+q)}|.
\end{align*}
Thus, it follows from \eqref{myq21} and \eqref{myq20} that
{  \begin{align*}
	\begin{split}
	&\exp\left\{{3p\gamma\tilde{\sigma}^2}e^{\lambda t}\left|\delta_{\theta}{Y}^{(m,q)}_t\right|\right\}\leq
	\exp\left\{{3p\gamma\tilde{\sigma}^2}e^{\lambda t}\left(\left( \delta_{\theta}{Y}^{(m,q)}_t\right)^++\left(\delta_{\theta}\widetilde{Y}^{(m,q)}_t\right)^++2|Y^{(m+q)}_t|\right)\right\}\\
	&\leq   \mathbb{\hat{E}}_t\bigg[\exp\bigg\{{3p\gamma\tilde{\sigma}^2}e^{\lambda T}\bigg(\rho(\theta,m)+|\xi|+\frac{\gamma T}{2}+|Y^{(m+q)}_t|+\int^T_t\left(\alpha_s+2\lambda|Y^{(m)}_s|+2\lambda|Y^{(m+q)}_s|\right)ds\bigg)\bigg\}	\bigg]^2
	\\ &\leq \mathbb{\hat{E}}_t\bigg[\exp\bigg\{6{p\gamma\tilde{\sigma}^2}e^{\lambda T}\bigg(\rho(\theta,m)+|\xi|+\frac{\gamma T}{2}+|Y^{(m+q)}_t|+\int^T_t\left(\alpha_s+2\lambda|Y^{(m)}_s|+2\lambda|Y^{(m+q)}_s|\right)ds\bigg)\bigg\}	\bigg],
	\end{split}
	\end{align*}
}
where we have used Jensen's inequality in the last inequality.

 Consequently, from Remark \ref{myq7} and  H\"{o}lder's inequality, we have
{ \small \begin{align}	\label{myq22}
	\begin{split}
	&\mathbb{\hat{E}}\bigg[\exp\left\{{3p\gamma\tilde{\sigma}^2}\sup\limits_{t\in[0,T]}\left|\delta_{\theta}{Y}^{(m,q)}_t\right|\right\}\bigg]
	\\ &\leq \hat{A}(G)\mathbb{\hat{E}}\bigg[\exp\bigg\{12{p\gamma\tilde{\sigma}^2}e^{\lambda T}\bigg(\rho(\theta,m)+|\xi|+\frac{\gamma T}{2}+(2\lambda T+1)\sup\limits_{t\in[0,T]}(|Y^{(m)}_t|+|Y^{(m+q)}_t|)+\int^T_0\alpha_tdt\bigg)\bigg\}	\bigg]\\
&\leq
\bar{A}(p) \mathbb{\hat{E}}\bigg[\exp\bigg\{\frac{24p{\gamma\tilde{\sigma}^2}e^{\lambda T}}{1-\theta}\bigg((|\xi|-m)^++\int^T_02(|f_0(t)|-m)^+dt\bigg)\bigg\}	\bigg]^{\frac{1}{2}},
	\end{split}
	\end{align}
}
where, in view of \eqref{myq11} and H\"{o}lder's inequality again,
{\small \[
\bar{A}(p):=\hat{A}(G)\sup\limits_{m,q\geq 1}\mathbb{\hat{E}}\bigg[\exp\bigg\{24{p\gamma\tilde{\sigma}^2}e^{\lambda T}\bigg(|\xi|+\frac{\gamma T}{2}+(2\lambda T+1)\sup\limits_{t\in[0,T]}(|Y^{(m)}_t|+|Y^{(m+q)}_t|)+\int^T_0\alpha_tdt\bigg)\bigg\}	\bigg]^{\frac{1}{2}}<\infty.
\]}

Next, we shall deal with the case that
when the generator $f$ is concave in $z$. We shall use
$\theta \ell^{(m+q)}- \ell^{(m)}$ and $\theta \ell^{(m)}- \ell^{(m+q)}$  instead of
$\ell^{(m+q)}-\theta \ell^{(m)}$ and $\ell^{(m)}-\theta \ell^{(m+q)}$  in the definition of $\delta_{\theta}\ell^{(m,q)}$ and $\delta_{\theta}\widetilde{\ell}^{(m,q)}$ for $\ell=Y,Z$, respectively.

In this case, the triple $(\delta_{\theta}Y^{(m,q)},\delta_{\theta}Z^{(m,q)},\frac{\theta}{1-\theta}K^{(m)})$
solves the following $G$-BSDE:
\begin{align*}
\begin{split}
\delta_{\theta}Y^{(m,q)}_t=&\delta_{\theta}\xi^{(m,q)}+\frac{1}{1-\theta}(K_T^{(m)}-K_t^{(m)})+\int^T_t\left(\delta_{\theta}f^{(m,q)}(s,\delta_{\theta}Y^{(m,q)}_s,\delta_{\theta}Z^{(m,q)}_s)+\delta_{\theta}f_0^{(m,q)}(s)\right)ds \\
&\ \ -\int^T_t\delta_{\theta}Z^{(m,q)}_sdB_s-\frac{\theta}{1-\theta}(K_T^{(m+q)}-K_t^{(m+q)})
\end{split}
\end{align*}
with
\begin{align*}
&\delta_{\theta}\xi^{(m,q)}=\frac{\theta\xi^{(m+q)}-\xi^{(m)}}{1-\theta},\ \ \ \delta_{\theta}f_0^{(m,q)}(t)=
\frac{1}{1-\theta}\left(\theta f^{(m+q)}_0(t)- f^{(m)}_0(t)\right)+f_0(t),\\
&
\delta_{\theta}f^{(m,q)}(t,y,z)=\frac{1}{1-\theta}\left(
\theta f(t,Y^{(m+q)}_t, Z^{(m+q)}_t)- f(t,-(1-\theta)y+\theta Y^{(m+q)}_t, -(1-\theta)z+\theta Z^{(m+q)}_t)\right).
\end{align*}
By Assumptions (H1) and (H3), it is easy to check that
\begin{align*}
&\delta_{\theta}\xi^{(m,q)}=-\xi^{(m+q)}+\frac{1}{1-\theta}(\xi^{(m+q)}-\xi^{(m)})\leq |\xi|+\frac{1}{1-\theta}(|\xi|-m)^+,\\
&\delta_{\theta}f_0^{(m,q)}(t)\leq -f_0^{(m+q)}(t)+f_0(t)+\frac{1}{1-\theta}(|f_0(t)|-m)^+\leq \frac{2}{1-\theta}(|f_0(t)|-m)^+,\\
&\delta_{\theta}f^{(m,q)}(t,y,z)
\leq  \lambda|y|+\lambda|Y^{(m+q)}_t|-f(t,Y^{(m+q)}_t, -z)
\leq  \alpha_t+\frac{\gamma}{2}+2\lambda|Y^{(m+q)}_t|+\lambda|y|+\frac{3\gamma}{2}|z|^2.
\end{align*}

Consequently, Inequalities \eqref{myq21} and  \eqref{myq20} should be replaced with
{ \small \begin{align*}	
	\begin{split}
	&\exp\left\{{3p\gamma\tilde{\sigma}^2}e^{\lambda t}\left(\delta_{\theta}Y^{(m,q)}_t\right)^+\right\}\leq   \mathbb{\hat{E}}_t\bigg[\exp\bigg\{{3p\gamma\tilde{\sigma}^2}e^{\lambda T}\bigg(\rho(\theta,m)+|\xi|+\frac{\gamma T}{2}+\int^T_t\left(\alpha_s+2\lambda|Y^{(m+q)}_s|\right)ds\bigg)\bigg\}	\bigg],\\
	&\exp\left\{{3p\gamma\tilde{\sigma}^2}e^{\lambda t}\left(\delta_{\theta}\widetilde{Y}^{(m,q)}_t\right)^+\right\}\leq   \mathbb{\hat{E}}_t\bigg[\exp\bigg\{3p{\gamma\tilde{\sigma}^2}e^{\lambda T}\bigg(\rho(\theta,m)+|\xi|+\frac{\gamma T}{2}+\int^T_t\left(\alpha_s+2\lambda|Y^{(m)}_s|\right)ds\bigg)\bigg\}	\bigg].
	\end{split}
	\end{align*}}
It follows from
$
\left(\delta_{\theta}{Y}^{(m,q)}\right)^-
\leq \left(\delta_{\theta}\widetilde{Y}^{(m,q)}\right)^++2|Y^{(m)}|
$ that
{  \begin{align*}
	\begin{split}
	&\exp\left\{{3p\gamma\tilde{\sigma}^2}e^{\lambda t}\left|\delta_{\theta}{Y}^{(m,q)}_t\right|\right\}
	\\ &\leq \mathbb{\hat{E}}_t\bigg[\exp\bigg\{6{p\gamma\tilde{\sigma}^2}e^{\lambda T}\bigg(\rho(\theta,m)+|\xi|+\frac{\gamma T}{2}+|Y^{(m)}_t|+\int^T_t\left(\alpha_s+2\lambda|Y^{(m)}_s|+2\lambda|Y^{(m+q)}_s|\right)ds\bigg)\bigg\}	 \bigg],
	\end{split}
	\end{align*}
}
and then Inequality \eqref{myq22} is still true when $f$ is concave in $z$.
	
{\bf 3.  The convergence.}
Since
$$\exp\bigg\{\frac{24p{\gamma\tilde{\sigma}^2}e^{\lambda T}}{1-\theta}\bigg((|\xi|-m)^++\int^T_02(|f_0(t)|-m)^+dt\bigg)\bigg\}\in L_G^{1}(\Omega)\downarrow 1 $$ as $m\rightarrow\infty$,
from  nonlinear monotone convergence theorem (Assertion (ii) of Lemma \ref{downward convergence proposition}), we have that for each $p\geq 1$ and $\theta\in (0,1)$,
\[
\lim\limits_{m\rightarrow\infty}\mathbb{\hat{E}}\bigg[\exp\bigg\{\frac{24p{\gamma\tilde{\sigma}^2}e^{\lambda T}}{1-\theta}\bigg((|\xi|-m)^++\int^T_02(|f_0(s)|-m)^+ds\bigg)\bigg\}	\bigg]^{\frac{1}{2}}=1,
\]
which together with Inequality \eqref{myq22} indicates that ,
\[
\limsup_{m\rightarrow \infty}\sup\limits_{q\geq 1}\mathbb{\hat{E}}\bigg[\exp\left\{{3p\gamma\tilde{\sigma}^2}\sup\limits_{t\in[0,T]}\left|\delta_{\theta}{Y}^{(m,q)}_t\right|\right\}\bigg]\leq \bar{A}(p), \ \forall \theta\in(0,1).
\]
It follows that for each $n\geq 1$  and $\theta\in (0,1)$,
\[
\limsup_{m\rightarrow \infty}\sup\limits_{q\geq 1}\mathbb{\hat{E}}\bigg[\sup\limits_{t\in[0,T]}
\left|\delta_{\theta}{Y}^{(m,q)}_t\right|^n\bigg]\leq \frac{\bar{A}(1)n!}{3^n\gamma^n \tilde{\sigma}^{2n}}.
\]
In view of the following fact
$${Y}^{(m+q)}-{Y}^{(m)}=(1-\theta)(\delta_{\theta}{Y}^{(m,q)}-Y^{(m)}) \ \text{ (resp. $(1-\theta)(\delta_{\theta}{Y}^{(m,q)}+Y^{(m+q)})$)} $$
  when $f$ is convex (resp. concave ) in $z$,  we derive that for any $n\geq 1$ and $\theta\in (0,1)$,
\begin{align*}
\limsup_{m\rightarrow \infty}\sup\limits_{q\geq 1}\mathbb{\hat{E}}\bigg[\sup\limits_{t\in[0,T]}
\left|{Y}^{(m+q)}_t-{Y}^{(m)}_t\right|^n\bigg]\leq 2^{n-1}(1-\theta)^n\bigg(\frac{\bar{A}(1)n!}{3^n\gamma^n \tilde{\sigma}^{2n}}+\sup\limits_{m\geq 1}\mathbb{\hat{E}}\bigg[\sup\limits_{t\in[0,T]}\left|{Y}^{(m)}_t\right|^n\bigg]\bigg).
\end{align*}
Sending $\theta\rightarrow 1$ and using \eqref{myq11}, we could find a continuous process $Y\in\mathcal{E}_G(\mathbb{R})$ such that
\begin{align}\label{myq37}
\lim_{m\rightarrow \infty}\mathbb{\hat{E}}\bigg[\sup\limits_{t\in[0,T]}
\left|{Y}^{(m)}_t-{Y}_t\right|^n\bigg]=0,\ \forall n\geq 1.
\end{align}
Indeed, from \eqref{myq11} and  Assertion (i) of Lemma \ref{downward convergence proposition}, we have that for any $p\geq 1$
\[
\mathbb{\hat{E}}\bigg[\exp\bigg\{{3p\gamma\tilde{\sigma}^2}\sup\limits_{0\leq t\leq T}|Y_t|\bigg\}
\bigg] \leq \hat{A}(G) \mathbb{\hat{E}}\bigg[\exp\bigg\{6p\gamma\tilde{\sigma}^2e^{\lambda T}\bigg(|\xi|+\frac{\gamma T}{2}+\int^T_0\alpha_sds\bigg)\bigg\}	\bigg].
\]

Now, applying $G$-It\^o's formula to $\left|{Y}^{(m+q)}_t-{Y}^{(m)}_t\right|^2$ yields that
\begin{align}\label{myq31}
\begin{split}
&\mathbb{\hat{E}}\bigg[\int^T_0\big|Z^{(m+q)}_t-Z^{(m)}_t\big|^2dt \bigg]\leq \tilde{\sigma}^2\mathbb{\hat{E}}\bigg[\int^T_0(Z^{(m+q)}_t-Z^{(m)}_t)(Z^{(m+q)}_t-Z^{(m)}_t)^{\top}d\langle B\rangle_t \bigg]\\
&\leq \tilde{\sigma}^2\mathbb{\hat{E}}\bigg[\sup\limits_{t\in[0,T]}
\left|{Y}^{(m+q)}_t-{Y}^{(m)}_t\right|^2+
\sup\limits_{t\in[0,T]}
\left|{Y}^{(m+q)}_t-{Y}^{(m)}_t\right|\Gamma^{(m,q)}\bigg]\\
&\leq \tilde{\sigma}^2\mathbb{\hat{E}}\bigg[\sup\limits_{t\in[0,T]}
\left|{Y}^{(m+q)}_t-{Y}^{(m)}_t\right|^2\bigg]+
\tilde{\sigma}^2\mathbb{\hat{E}}[|\Gamma^{(m,q)}|^2]^{\frac{1}{2}}\mathbb{\hat{E}}\bigg[\sup\limits_{t\in[0,T]}
\left|{Y}^{(m+q)}_t-{Y}^{(m)}_t\right|^2\bigg]^{\frac{1}{2}},
\end{split}
\end{align}
with
\[
\Gamma^{(m,q)}:=\int^T_0\big|f^{(m+q)}(t,Y^{(m+q)}_t,Z^{(m+q)}_t)- f^{(m)}(t,Y^{(m)}_t,Z^{(m)}_t)\big|dt+|K_T^{(m+q)}|+|K_T^{(m)}|.
\]
In spirit of  Inequalities \eqref{myq40}, \eqref{myq11}, and \eqref{myq30}, we see that
$\sup\limits_{m,q\geq 1}\mathbb{\hat{E}}[|\Gamma^{(m,q)}|^2]<\infty. $
Thus, by   \eqref{myq37} and \eqref{myq31}, there is a process $Z\in M^2_G(0,T)$ so that
\begin{align}\label{myq51}
\lim_{m\rightarrow \infty}\mathbb{\hat{E}}\bigg[\int^T_0\big|Z^{(m)}_t-Z_t\big|^2dt \bigg]=0.
\end{align}
In view of  Assertion (i) of Lemma \ref{downward convergence proposition} and \eqref{myq30}, we have that $Z\in\mathcal{H}_G(\mathbb{R}^d)$, i.e., for any $n\geq 1$
\[
\mathbb{\hat{E}}\bigg[\bigg(\int^T_0|Z_t|^2dt\bigg)^n \bigg]<\infty,
\]
which together with Lemma \ref{myq38} and   \eqref{myq30}, \eqref{myq51} implies that
\begin{align}\label{myq511}\lim_{m\rightarrow \infty}\mathbb{\hat{E}}\bigg[\bigg(\int^T_0\big|Z^{(m)}_t-Z_t\big|^2dt\bigg)^n \bigg]=0,\ \forall n\geq 1.
\end{align}

On the other hand, from Assumption (H1), we have
\begin{align*}
|f^{(m)}(t, Y_t^{(m)}, Z_t^{(m)})-f(t, Y_t, Z_t)|\leq\lambda| Y_t^{(m)}-Y_t|+{\gamma}(1+|Z_t^{(m)}|+|Z_t|)|Z_t^{(m)}-Z_t|+(|f_0(t)|-m)^+.
\end{align*}
Then,  from  H\"{o}lder's inequality, we have  for each $n\geq 1$,
\begin{align*}
&\mathbb{\hat{E}}\bigg[\bigg(\int^T_0|f^{(m)}(t, Y_t^{(m)}, Z_t^{(m)})-f(t, Y_t, Z_t)|dt\bigg)^n\bigg]\\
&\leq 3^{n-1}\lambda^nT^n \mathbb{\hat{E}}\bigg[\sup\limits_{t\in[0,T]}
\left|{Y}^{(m)}_t-{Y}_t\right|^n\bigg]+ 3^{n-1}\mathbb{\hat{E}}\bigg[\bigg(\int^T_0(|f_0(t)|-m)^+dt\bigg)^n\bigg].\\ &\ \ \ \ +3^{n-1}\gamma^n\mathbb{\hat{E}}\bigg[\bigg(\int^T_0(1+|Z_t^{(m)}|+|Z_t|)^2dt\bigg)^n \bigg]^{\frac{1}{2}}\mathbb{\hat{E}}\bigg[\bigg(\int^T_0\big|Z^{(m)}_t-Z_t\big|^2dt\bigg)^n \bigg]^{\frac{1}{2}},
\end{align*}
which converges to $0$ as $m\rightarrow\infty$ in view of Equations
\eqref{myq37} and \eqref{myq511} and Assertion (ii) of Lemma \ref{downward convergence proposition}.
We set
\[
K_t=Y_t-Y_0+\int_0^tf(s,Y_s,Z_s)\, ds-\int_0^tZ_s\, dB_s.
\]
Therefore,  $\mathbb{\hat{E}}\left[\left|K_t-K_t^{(m)}\right|^n\right]\rightarrow 0$ for each $n\geq 1$. Thus, $K$ is a non-increasing $G$-martingale, and
then $(Y,Z,K)\in \mathcal{E}_G(\mathbb{R})\times\mathcal{H}_G(\mathbb{R}^d)\times \mathcal{L}_G(\mathbb{R})$ satisfies Equation \eqref{myq1}. The proof is complete.
\end{proof}

\begin{lemma}\label{myq38}
Let $X_n\in L^1_G(\Omega)$ for $n\geq 1$ such that $
\sup\limits_{n\geq 1}\mathbb{\hat{E}}[|X_n|^{2p}]<\infty \ \text{for some $p\geq 1$}.
$
If $\mathbb{\hat{E}}[|X_n|]$ converges to $0$ as $n \rightarrow \infty$, then
$
\lim\limits_{n\rightarrow\infty}\mathbb{\hat{E}}[|X_n|^p]=0.
$
\end{lemma}

\begin{proof}
For any $\varepsilon>0$, we have
\begin{align*}
\mathbb{\hat{E}}[|X_n|^p]\leq \varepsilon^p+\mathbb{\hat{E}}[|X_n|^p\mathbf{1}_{\{|X_n|>\varepsilon\}}]
\leq \varepsilon^p+{\varepsilon}^{-{\frac{1}{2}}}\mathbb{\hat{E}}[|X_n|^{2p}]^{\frac{1}{2}}\mathbb{\hat{E}}[|X_n|]^{\frac{1}{2}}.
\end{align*}
Therefore,  we have
$$
\limsup\limits_{n\rightarrow\infty}\mathbb{\hat{E}}[|X_n|^p]\leq \varepsilon^p.
$$
 Sending $\varepsilon\rightarrow 0$, we complete the proof.
\end{proof}

\section{Multi-dimensional quadratic $G$-BSDEs}

In this section, we  consider multi-dimensional quadratic $G$-BSDEs on time interval $[0,T]$:
\begin{align}\label{myq1213}
Y_{t}=\xi+\int_{t}^{T}f(s,Y_{s},Z_{s})ds-\int_{t}^{T}Z_{s}dB_{s}-(K_{T}-K_{t}),
\end{align}
	where the generators
\[
 f(t,\omega,y,z)=(f^1(t,\omega,y,z),\cdots,f^n(t,\omega,y,z))^{\top}: [0,T]\times\Omega\times
\mathbb{R}^n\times\mathbb{R}^{n\times d}\rightarrow\mathbb{R}^n.
\]

For sake of convenience,  denote by $y^l$ and $z^l$ the $l$-th component of $y$ and the $l$-th row of $z$ for each argument $(y,z)\in\mathbb{R}^n\times\mathbb{R}^{n\times d}$, respectively. Consider the following assumptions.

\begin{description}
	\item[(B1)] For each $l=1,\cdots,n$, $f^l(t,\omega,y,z)$ depends only  on the $l$-th row $z^l$  of the
	argument $z$ and is convex or concave in $z^l$.
\item[(B2)] For each  $(t,\omega)\in[0,T]\times\Omega$ and $(y,z),(\bar{y},\bar{z})\in \mathbb{R}^n\times \mathbb{R}^{n\times d}$,
	\[
|f(t,\omega,0,0)|\leq \alpha_t(\omega)\ \text{and}\	|f(t,\omega,y,z)-f(t,\omega,\bar{y},\bar{z})|\leq \lambda|y-\bar{y}|+{\gamma}(1+|z|+|\bar{z}|)|z-\bar{z}|.
\]
\item[(B3)] There exists a modulus of continuity $w:[0,\infty)\rightarrow[0,\infty)$  such that for each $(y,z)\in \mathbb{R}^n\times \mathbb{R}^{n\times d}$ and $(t,\omega), (\bar{t},\bar{\omega})
\in [0,T]\times \Omega$,
\[
|f(t,\omega,y,z)-f(\bar{t},\bar{\omega},y,z)|\leq w(|t-\bar{t}|+\|\omega-\bar{\omega}\|).
\]
\item[(B4)] Both the terminal value $\xi\in L^1_G(\Omega;\mathbb{R}^n)$  and $\int^T_0\alpha_tdt$ have  exponential moments of arbitrary order, i.e.,
\[
\mathbb{\hat{E}}\bigg[\exp\bigg\{p|\xi|+p\int^T_0\alpha_tdt\bigg\}\bigg]<\infty\quad  \text{for any $p\geq 1$.}
\]	
\end{description}

\begin{lemma}\label{myq79}
	Assume that all Assumptions (B1)-(B4) hold and $U\in\mathcal{E}_G(\mathbb{R}^n)$. Then, the following multi-dimensional decoupled $G$-BSDE on $[0,T]$:
	\begin{align*}
	Y^{l}_{t}=\xi^l+\int_{t}^{T}f^l(s,U_{s},Z^{l}_{s})ds-\int_{t}^{T}Z^{l}_{s}dB_{s}-(K^{l}_{T}-K^{l}_{t}),\ \forall l=1,\ldots,n,
	\end{align*}
	admits a unique solution $(Y,Z,K)\in \mathcal{E}_G(\mathbb{R}^n)\times\mathcal{H}_G(\mathbb{R}^{n\times d})\times\mathcal{L}_G(\mathbb{R}^n)$.
\end{lemma}
\begin{proof}
	From Assumptions (B1)-(B3), we have that for $l=1,\ldots,n$,
	\begin{align}\label{myq82}
	|f^l(t,U_{t},z^l)|\leq \alpha_t+\frac{\gamma }{2}+\lambda|U_{t}|+\frac{3\gamma}{2}|z^l|^2.
	\end{align}
	From Assumption (B4) and  H\"{o}lder's inequality, we have	
	\[
	\mathbb{\hat{E}}\bigg[\exp\bigg\{p|\xi^l|+p\int^T_0\left(\alpha_t+\lambda |U_t|\right)dt\bigg\}\bigg]<\infty\ \text{for each $p\geq 1$.}
	\]
	Consequently, applying Theorem \ref{myq66}, we have  the desired result.
\end{proof}

Now, following the idea of \cite{FHT2}, we  study the well-posedness of solutions to multi-dimensional quadratic $G$-BSDE \eqref{myq1213} of diagonally quadratic
generators.

\begin{theorem}\label{myq79}
	Assume that all Assumptions (B1)-(B4) are satisfied. Then, the  multi-dimensional  $G$-BSDE \eqref{myq1213}	admits a unique solution $(Y,Z,K)\in \mathcal{E}_G(\mathbb{R}^n)\times\mathcal{H}_G(\mathbb{R}^{n\times d})\times\mathcal{L}_G(\mathbb{R}^n)$.
\end{theorem}

\begin{proof}
With the help of Lemma \ref{myq79}, the iterative method in the proof of \cite[Theorem 2.8]{FHT2} still works here. Indeed,
 we firstly set $Y^{(0)}=0$, and  define recursively the sequence of stochastic processes $(Y^{(m)})_{m=1}^{\infty}$ through solution of the following $G$-BSDE on $[0,T]$:
\begin{align}\label{myq78}
Y^{(m);l}_{t}=\xi^l+\int_{t}^{T}f^l(s,Y^{(m-1)}_{s},Z^{(m);l}_{s})ds-\int_{t}^{T}Z^{(m);l}_{s}dB_{s}-(K^{(m);l}_{T}-K^{(m);l}_{t}),\ \forall l=1,\ldots,n.
\end{align}
From Lemma \ref{myq79}, we  get that $(Y^{(m)},Z^{(m)},K^{(m)})\in \mathcal{E}_G(\mathbb{R}^n)\times\mathcal{H}_G(\mathbb{R}^{n\times d})\times\mathcal{L}_G(\mathbb{R}^n)$. Next, we use Lemma \ref{myq2} to establish a uniform estimate on $(Y^{(m)},Z^{(m)},K^{(m)})$, and then utilize a $\theta$-method to get the convergence of $Y^{(m)}$ and the uniqueness without  any further difficulty.
For the reader's convenience, we  sketch the proof.

In view of \eqref{myq82} and Assertion (i) of Lemma \ref{myq2} (taking  $\beta_t=\alpha_t+\frac{\gamma}{2}+\lambda |Y^{(m-1)}_t|$, $\lambda=0$  and $\kappa=3\gamma$), we have  that  for any $p\geq 1$ and $l=1,\ldots,n$,
\begin{align*}	
	\begin{split}
	\exp\left\{{3p\gamma\tilde{\sigma}^2}\left|Y^{(m);l}_t\right|\right\}\leq   \mathbb{\hat{E}}_t\bigg[\exp\bigg\{{3p\gamma\tilde{\sigma}^2}\bigg(|\xi|+\int^T_t\left(\alpha_s+\frac{\gamma }{2}+\lambda|Y^{(m-1)}_s|\right)ds\bigg)\bigg\}	\bigg],\ \forall m\geq 1.
	\end{split}
	\end{align*}
	From Jensen's inequality, we have
\begin{align*}	
\begin{split}
\exp\left\{{3p\gamma\tilde{\sigma}^2}\left|Y^{(m)}_t\right|\right\}\leq   \mathbb{\hat{E}}_t\bigg[\exp\bigg\{{3np\gamma\tilde{\sigma}^2}\bigg(|\xi|+\int^T_t\left(\alpha_s+\frac{\gamma }{2}+\lambda|Y^{(m-1)}_s|\right)ds\bigg)\bigg\}	 \bigg],\ \forall m\geq 1.
\end{split}
\end{align*}
In spirit of  Remark \ref{myq7}, we get that for any $p\geq 1$, $m\geq 1$ and $t\in[0,T]$,
{\small \begin{align}
\label{myq83}
\begin{split}
\mathbb{\hat{E}}\bigg[\exp\bigg\{{3p\gamma\tilde{\sigma}^2}\sup\limits_{t\leq s\leq T}|Y^{(m)}_s|\bigg\}
\bigg] &\leq \hat{A}(G)  \mathbb{\hat{E}}\bigg[\exp\bigg\{{6np\gamma\tilde{\sigma}^2}\bigg(|\xi|+\int^T_t\left(\alpha_s+\frac{\gamma }{2}+\lambda |Y^{(m-1)}_s|\right)ds\bigg)\bigg\}	\bigg]\\
&\leq  \sqrt{\underline{A}(p)} \mathbb{\hat{E}}\bigg[\exp\bigg\{{12np\gamma\tilde{\sigma}^2}\lambda(T-t) \sup\limits_{t\leq s\leq T}|Y^{(m-1)}_s|\bigg\}	\bigg]^{\frac{1}{2}}
\end{split}
\end{align}	}
with
\[
\underline{A}(p)= |\hat{A}(G)|^2\mathbb{\hat{E}}\bigg[\exp\bigg\{{12np\gamma\tilde{\sigma}^2}\bigg(|\xi|+\int^T_t\big(\alpha_s+\frac{\gamma }{2}\big)ds\bigg)\bigg\}	\bigg]<\infty.
\]

Define \begin{align*}
\mu:=
\begin{cases} 4n\lambda T, \  &\text{if $4n\lambda T$ is an integer};\\
 [4n\lambda T]+1, \ &\text{otherwise}.
\end{cases}
\end{align*}
If $\mu=1$, it follows from \eqref{myq83} that for each $p\geq 1$ and $m\geq 1$
\begin{align*}
\begin{split}
\mathbb{\hat{E}}\bigg[\exp\bigg\{{3p\gamma\tilde{\sigma}^2}\sup\limits_{0\leq s\leq T}|Y^{(m)}_s|\bigg\}
\bigg] \leq (\underline{A}(p))^{\frac{1}{2}} \mathbb{\hat{E}}\bigg[\exp\bigg\{{3p\gamma\tilde{\sigma}^2} \sup\limits_{0\leq s\leq T}|Y^{(m-1)}_s|\bigg\}	 \bigg]^{\frac{1}{2}},
\end{split}
\end{align*}	
which implies that
\begin{align}\label{myq698}
\begin{split}
&\mathbb{\hat{E}}\bigg[\exp\bigg\{{3p\gamma\tilde{\sigma}^2}\sup\limits_{0\leq s\leq T}|Y^{(m)}_s|\bigg\}
\bigg] \leq (\underline{A}(p))^{\frac{1}{2}+\frac{1}{4}+\frac{1}{2^m}} \mathbb{\hat{E}}\bigg[\exp\bigg\{{3p\gamma\tilde{\sigma}^2} \sup\limits_{0\leq s\leq T}|Y^{(0)}_s|\bigg\}	\bigg]^{\frac{1}{2^m}}\\
& \ \ \ \ \ \ \leq\underline{A}(p)\leq |\hat{A}(G)|^2\mathbb{\hat{E}}\bigg[\exp\bigg\{{24np\gamma\tilde{\sigma}^2}|\xi|\bigg\}	 \bigg]\mathbb{\hat{E}}\bigg[\exp\bigg\{{24np\gamma\tilde{\sigma}^2}\int^T_0\big(\alpha_s+\frac{\gamma }{2}\big)ds\bigg\}	\bigg].
\end{split}
\end{align}
If $\mu=2$, proceeding identically as in the above, we have for any $p\geq 1$,
 \begin{align}\label{myq699}
 \begin{split}
&\mathbb{\hat{E}}\bigg[\exp\bigg\{{3p\gamma\tilde{\sigma}^2}\sup\limits_{T-(4n\lambda)^{-1}\leq s\leq T}|Y^{(m)}_s|\bigg\}
\bigg] \\
&\leq |\hat{A}(G)|^2\mathbb{\hat{E}}\bigg[\exp\bigg\{{24np\gamma\tilde{\sigma}^2}|\xi|\bigg\}	 \bigg]\mathbb{\hat{E}}\bigg[\exp\bigg\{{24np\gamma\tilde{\sigma}^2}\int^T_0\big(\alpha_s+\frac{\gamma }{2}\big)ds\bigg\}	\bigg].
\end{split}
\end{align}

Then, consider the following $G$-BSDEs on time interval $[0,T-(4n\lambda)^{-1}]$ for each $m\geq 1$:
{\small \begin{align*}
Y^{(m);l}_{t}=Y^{(m);l}_{T-(4n\lambda)^{-1}}+\int_{t}^{T-(4n\lambda)^{-1}}f^l(s,Y^{(m-1)}_{s},Z^{(m);l}_{s})ds-\int_{t}^{T-(4n\lambda)^{-1}}Z^{(m);l}_{s}dB_{s}-(K^{(m);l}_{T-(4n\lambda)^{-1}}-K^{(m);l}_{t}).
\end{align*}
}
Proceeding identically as to derive \eqref{myq698}, we have
{\small \begin{align*}
	\begin{split}
	&\mathbb{\hat{E}}\bigg[\exp\bigg\{{3p\gamma\tilde{\sigma}^2}\sup\limits_{0\leq s\leq T-(4n\lambda)^{-1}}|Y^{(m)}_s|\bigg\}
	\bigg]\\ &\leq |\hat{A}(G)|^2\mathbb{\hat{E}}\bigg[\exp\bigg\{{24np\gamma\tilde{\sigma}^2}|Y^{(m)}_{T-(4n\lambda)^{-1}}|\bigg\}	 \bigg]\mathbb{\hat{E}}\bigg[\exp\bigg\{{24np\gamma\tilde{\sigma}^2}\int^T_0\big(\alpha_s+\frac{\gamma }{2}\big)ds\bigg\}	\bigg]\\
	&\leq |\hat{A}(G)|^4\mathbb{\hat{E}}\bigg[\exp\bigg\{{192n^2p\gamma\tilde{\sigma}^2}|\xi|\bigg\}	 \bigg]\mathbb{\hat{E}}\bigg[\exp\bigg\{{384n^2p\gamma\tilde{\sigma}^2}\int^T_0\big(\alpha_s+\frac{\gamma }{2}\big)ds\bigg\}	 \bigg],
	\end{split}
	\end{align*}
}where we have used \eqref{myq699} in the last inequality.
From the last two inequalities and H\"{o}lder's inequality, we get
{\small \begin{align*}
	\begin{split}
	\mathbb{\hat{E}}\bigg[\exp\bigg\{{3p\gamma\tilde{\sigma}^2}\sup\limits_{0\leq s\leq T}|Y^{(m)}_s|\bigg\}
	\bigg]
	&\leq |\hat{A}(G)|^3\mathbb{\hat{E}}\bigg[\exp\bigg\{{384n^2p\gamma\tilde{\sigma}^2}|\xi|\bigg\}	 \bigg]\mathbb{\hat{E}}\bigg[\exp\bigg\{{768n^2p\gamma\tilde{\sigma}^2}\int^T_0\big(\alpha_s+\frac{\gamma }{2}\big)ds\bigg\}	 \bigg].
	\end{split}
	\end{align*}
}
Iterating the above procedure $\mu$ times, we have
{\small \begin{align}\label{myq811}
	\begin{split}
	&\mathbb{\hat{E}}\bigg[\exp\bigg\{{3p\gamma\tilde{\sigma}^2}\sup\limits_{0\leq s\leq T}|Y^{(m)}_s|\bigg\}
	\bigg]\\
	&\leq |\hat{A}(G)|^{\mu+1}\mathbb{\hat{E}}\bigg[\exp\bigg\{{24n(16n)^{\mu-1}p\gamma\tilde{\sigma}^2}|\xi|\bigg\}	 \bigg]\mathbb{\hat{E}}\bigg[\exp\bigg\{{24n(32n)^{\mu-1}p\gamma\tilde{\sigma}^2}\int^T_0\big(\alpha_s+\frac{\gamma }{2}\big)ds\bigg\}	\bigg].
	\end{split}
	\end{align}
}
Furthermore, in view of \eqref{myq82}, using Lemma \ref {myq299}, we see that for  $p\geq 1$ and $l=1,\ldots,n$,	
 \begin{align*}&\mathbb{\hat{E}}\left[\bigg(\int^T_0|Z_t^{(m);l}|^2dt\bigg)^p+|K^{(m);l}_T|^p \right]\\
 &\leq \tilde{A}(p)\mathbb{\hat{E}}\bigg[\exp\bigg\{12\gamma\tilde{\sigma}^2n\sup\limits_{t\in[0,T]}|Y^{(m)}_t|+2\lambda n T\sup\limits_{t\in[0,T]}|Y^{(m-1)}_t|+2n\int^T_0\big(\alpha_t+\frac{\gamma}{2}\big)dt\bigg\}\bigg],
\end{align*}
which is uniformly bounded with respect to $m$.

Finally, in view of Assertion (ii) of Lemma \ref{myq2}, the proof of Theorem \ref{myq66} and the above derivation of \eqref{myq811}, proceeding  identically to  that of  \cite[Theorem 2.8]{FHT2}, we complete the proof.
\end{proof}

\appendix
\renewcommand\thesection{\normalsize Appendix }
\section{ }

\renewcommand\thesection{A}
\normalsize

In this appendix, we introduce the extended conditional $G$-expectation, which are needed in this paper.
 Consider the following two spaces of random variables
\begin{align*}
\mathbb{L}^1(\Omega):=\{X\in \mathcal{B}(\Omega):\hat{\mathbb{E}}[|X|]<\infty\},\text{and} \
L_G^{1^*}(\Omega):=\{X\in\mathbb{L}^1(\Omega):\exists X_n\in L_G^1(\Omega) \textrm{ such that } X_n\downarrow X\}.
\end{align*}
 Then we could extend the conditional $G$-expectation to the space $L_G^{1^*}(\Omega_t)$:
 \[
 \mathbb{\hat{E}}_t[X]=\lim\limits_{n\rightarrow\infty}\mathbb{\hat{E}}_t[X_n],
 \]
 which does not depend on  the choice of approximating sequences, see \cite{HP13} for more details.

\begin{lemma}\emph{(Proposition 35 in \cite{HP13})}
Assume that $X\in L_G^{1^*}(\Omega)$. Then,
for any $P\in\mathcal{P}$
\[
\mathbb{\hat{E}}_t[X]=\esssup\limits_{\bar{P}\in\mathcal{P}(t,P)}E^{\bar{P}}_t[X]\  \text{$P$-a.s. for any  $ t\geq 0$},
\]
where $\mathcal{P}(t,P)=\{\bar{P}\in\mathcal{P}|\ \bar{P}=P \text{ on $Lip(\Omega_t)$} \}$.\label{myq9}
\end{lemma}

\end{document}